\documentclass[oneside,12pt]{amsart}

\usepackage{amscd,amsmath,latexsym}
\usepackage[mathcal]{euscript}
\newtheorem{thm}{Theorem}[section]
\newtheorem{cor}[thm]{Corollary}
\newtheorem{lem}[thm]{Lemma}
\newtheorem{prop}[thm]{Proposition}

\theoremstyle{definition}
\newtheorem{defin}[thm]{Definition}
\theoremstyle{definition}
\newtheorem{exm}[thm]{Example}

\newtheorem{remark}[thm]{Remark}
\theoremstyle{conjecture}
\newtheorem{conjecture}[thm]{Conjecture}
\theoremstyle{remark}

\newcommand{\R}{{\mathbb R}}
\newcommand{\C}{{\mathbb C}}

\def\C{{\mathbb C}}
\def\R{{\mathbb R}}

\def\z2{{\bf Z}$_{2}$}
\def\zl2{{\bf Z}$_{(2)}$}

\def\ds{\displaystyle}
\begin{document}
\def\s{\sigma}

\title[The polyhedral product functor]{The polyhedral product functor: a method  of computation for
moment-angle complexes, arrangements and related spaces}

\author[A.~Bahri]{A.~Bahri}
\address{Department of Mathematics,
Rider University, Lawrenceville, NJ 08648, U.S.A.}
\email{bahri@rider.edu}

\author[M.~Bendersky]{M.~Bendersky}
\address{Department of Mathematics
CUNY,  East 695 Park Avenue New York, NY 10065, U.S.A.}
\email{mbenders@xena.hunter.cuny.edu}

\author[F.~R.~Cohen]{F.~R.~Cohen}
\address{Department of Mathematics,
University of Rochester, Rochester, NY 14625, U.S.A.}
\email{cohf@math.rochester.edu}

\author[S.~Gitler]{S.~Gitler}
\address{Department of Mathematics,
Cinvestav, San Pedro Zacatenco, Mexico, D.F. CP 07360 Apartado
Postal 14-740, Mexico} \email{sgitler@math.cinvestav.mx}



\begin{abstract}

This article gives a natural decomposition of the suspension of
generalized moment-angle complexes or {\it partial product spaces}
which arise as {\it polyhedral product functors } described below.

In the special case of the complements of certain subspace
arrangements, the geometrical decomposition implies the homological
decomposition in Goresky-MacPherson \cite{goresky.macpherson},
Hochster\cite{hochster}, Baskakov \cite{baskakov}, Panov
\cite{panov}, and Buchstaber-Panov \cite{buchstaber.panov}. Since
the splitting is geometric, an analogous homological decomposition
for a generalized moment-angle complex applies for any homology
theory. This decomposition gives an additive decomposition for the
Stanley-Reisner ring of a finite simplicial complex and
generalizations of certain homotopy theoretic results of Porter
\cite{porter} and Ganea \cite{ganea}.  The spirit of the work here
follows that of Denham-Suciu in \cite{denham.suciu}.

\end{abstract}

\maketitle

\section{Introduction}

Spaces which now are called {\em (generalized) moment-angle
complexes\/}, have been studied by topologists since the 1960's
thesis of G.~Porter \cite{porter}. In the 1970's E.~B.~Vinberg
\cite{vinberg} and in the late 1980's S.~Lopez~de~Medrano developed
some of their features \cite{Santiago}. In seminal work during the
early 1990's, M.~Davis and J.~Januszkiewicz introduced quasi-toric
manifolds, a topological generalization of projective toric
varieties which were being studied intensively by algebraic
geometers \cite{davis.jan}. They observed that every quasi-toric
manifold is the quotient of a moment-angle complex by the free
action of a real torus. The moment-angle complex is denoted
$Z(K;(D^2,S^1))$.

The integral cohomology of the spaces $Z(K;(D^2,S^1))$ has been
studied by Goresky-MacPherson \cite{goresky.macpherson}, Hochster
\cite{hochster}, Baskakov \cite{baskakov}, Panov \cite{panov},
Buchstaber-Panov \cite{buchstaber.panov} and Franz \cite{franz,
franz2}. Among others who have worked extensively on generalized
moment-angle complexes are Notbohm-Ray \cite{notbohm.ray},
Grbic-Theriault \cite{grbic.theriault}, Strickland \cite{strickland}
and Kamiyama-Tsukuda \cite{kamiyama.tsukuda}. The direction of this
paper is guided by the development in elegant work of Denham-Suciu
\cite{denham.suciu}.

Among the results given here is a natural decomposition for the
suspension of the generalized moment-angle complex, the value of the
suspension of the ``polyhedral product functor". Since the
decomposition is geometric, an analogous homological decomposition
for a generalized moment-angle complex applies for any homology
theory. This last decomposition specializes to the homological
decompositions in the work of several authors cited above.
Furthermore, this decomposition gives an additive decomposition for
the Stanley-Reisner ring of a finite simplicial complex extended to
other natural settings and which arises from generalizations of
certain homotopy theoretic results of Porter \cite{porter} and Ganea
\cite{ganea}. Extensions of these structures to a simplicial setting
as well as to the structure of ordered, commuting $n$-tuples in a
Lie group $G$ are given in \cite{abbcg}.

Generalized moment-angle complexes are given by functors determined
by simplicial complexes with values given by subspaces of products
of topological spaces. Thus, these may be regarded as ``polyhedral
product functors" indexed by abstract simplicial complexes,
terminology for which the authors thank Bill Browder (also the
inventor of the name ``orbifold").

\tableofcontents

\section{Statement of results}\label{results}

In this article, a polyhedron is defined to be the geometric
realization of a simplicial complex. Generalized moment-angle
complexes can be be regarded loosely as a functor from simplicial
complexes with values given by subspaces of products of topological
spaces. Thus these may be regarded as "polyhedral product functors".
Generalized moment-angle complexes are defined next.

\begin{enumerate}

\item Let $(\underline{X},\underline{A})=\{(X_i,A_i,x_i)^m_{i=1}\}$ denote
a set of triples of $CW$--complexes with base-point $x_i$ in $A_i$.

\item Let $K$ denote an abstract simplicial complex with $m$ vertices labeled by the set
\newline $[m]=\{1,2,\ldots, m\}$. Thus, a $(k-1)$-simplex $\sigma$ of $K$ is given by
an ordered sequence $\sigma =(i_1,\ldots, i_k)$ with $1 \leq i_1
<\ldots < i_k \leq m$ such that if $\tau \subset \sigma$, then
$\tau$ is a simplex of $K$. In particular the empty set $\phi$ is a
subset of $\sigma$ and so it is in $K$. The set $[m]$ is minimal in
the sense that every $i \in [m]$ belongs to at least one simplex.
The length of of $\sigma=k$ is denoted $|\sigma|$.

\item Let $\Delta[m-1]_q$ denote the $q$-skeleton of $(m-1)$-simplex having $m$ vertices.
That is, $\Delta[m-1]_q$ denotes the subsets of $[m]$ of cardinality
at most $q+1$.
\end{enumerate}

Define two functors from the category $K$ of simplicial complexes
with morphisms simplicial maps to the category $CW_{\ast}$ of
connected, based $CW$-complexes and based continuous maps as
follows.
\begin{defin}\label{defin:moment.angle.complex}
As above, let $(\underline{X},\underline{A})$ denote the collection
$\{(X_i, A_i,x_i)\}^m_{i=1}$.

\item The {\it generalized moment-angle complex or polyhedral product functor } determined by
$\ds{(\underline{X},\underline{A})}$  and $K$ denoted
$$Z(K;(\underline{X},\underline{A}))$$ is defined using the functor $$D: K \to
CW_{\ast}$$ as follows: For every $\sigma$ in $K$, let
$$
D(\sigma) =\prod^m_{i=1}Y_i,\quad {\rm where}\quad
Y_i=\left\{\begin{array}{lcl}
X_i &{\rm if} & i\in \sigma\\
A_i &{\rm if} & i\in [m]-\sigma.
\end{array}\right.$$ with $D(\emptyset) = A_1 \times \ldots \times A_m$.

The generalized moment-angle complex is
$$Z(K;(\underline{X},\underline{A}))=\bigcup_{\sigma \in K}
D(\sigma)= \mbox{colim} D(\sigma)$$ where the colim is defined by
the inclusions, $d_{\sigma,\tau}$ with $\sigma \subset \tau$ and
$D(\sigma)$ is topologized as a subspace of the product $X_1 \times
\ldots \times X_m$. The {\it generalized pointed moment-angle} is
the underlying space $Z(K;(\underline{X},\underline{A}))$ with
base-point $\underline{*} = (x_1, \ldots, x_m) \in
Z(K;(\underline{X},\underline{A}))$.

Note that the definition of $Z(K;(\underline{X},\underline{A}))$ did
not require spaces to be either based or $CW$ complexes. In the
special case where $X_i = X$ and $A_i = A$ for all $1 \leq i \leq
m$, it is convenient to denote the generalized moment-angle complex
by $Z(K;(X,A))$ to coincide with the notation in
\cite{denham.suciu}.
\end{defin}

The smash product $$X_1 \wedge X_2 \wedge \ldots \wedge X_m$$ is
given by the quotient space $(X_1 \times \ldots \times X_m) /S(X_1
\times \ldots \times X_m)$ where $S(X_1 \times \ldots \times X_m)$
is the subspace of the product with at least one coordinate given by
the base-point $x_j \in X_j$. Spaces analogous to generalized
moment-angle complexes are given next where products of spaces are
replaced by smash products, a setting in which base-points are
required.

\begin{defin}\label{defin:smash.product.moment.angle.complex}
Given a generalized pointed-moment-angle complex
$Z(K;(\underline{X},\underline{A}))$ obtained from
$(\underline{X},\underline{A}, \underline{*})$, the {\it generalized
smash moment-angle complex}
$$\widehat{Z}(K;(\underline{X},\underline{A}))$$ is defined to be the image of
$Z(K;(\underline{X},\underline{A}))$ in the smash product $X_1
\wedge X_2 \wedge \ldots \wedge X_m$.

The image of $D(\sigma)$ in $\widehat{Z}(K;
(\underline{X},\underline{A}))$ is denoted by $\widehat{D}(\sigma)$
and is $$Y_1 \wedge Y_2 \wedge \ldots \wedge Y_m$$ where
$$Y_i=\left\{\begin{array}{lcl}
X_i &{\rm if} & i\in \sigma\\
A_i &{\rm if} & i\in [m]-\sigma.
\end{array}\right.
$$
\end{defin}

As in the case of $Z(K;(\underline{X},\underline{A}))$, note that
$\widehat{Z}(K;(\underline{X},\underline{A}))$ is the colimit
obtained from the spaces $ \widehat{D}(\sigma) \mbox{ with }
\widehat{D}(\sigma) \cap \widehat{D}(\tau) = \widehat{D}(\sigma \cap
\tau).$

\begin{remark}\label{remark:naturality}
The constructions $Z(K;(\underline{X},\underline{A}))$ and
$\widehat{Z}(K;(\underline{X},\underline{A}))$ are bifunctors,
natural for morphisms of $(\underline{X},\underline{A})$ as well as
$(\underline{X},\underline{A}, \underline{*})$, and for injections
in $K$.
\end{remark}

The next example provides an elementary case of a moment-angle
complex together with an associated toric manifold.
\begin{exm} \label{exm:example.with.two.points}
Let $K$ denote the simplicial complex given by $\{ \{1\}, \{2\},
\emptyset \}$ with $$(X_i,A_i)=(D^{n_i},S^{n_i-1})\quad{\rm
for}\quad i=1,2,$$ and let $K'$ denote an arbitrary simplicial
complex with $m$ vertices.

\begin{enumerate}

\item The first example is $$(X_i,A_i)=(D^{2n},S^{2n-1})\quad{\rm for}\quad i=1,2$$
where $D^{2n}$ is the $2n$-disk and $S^{2n-1}$ is its boundary
sphere. Then $$Z(K;(D^{2n},S^{2n-1}))=D^{2n}\times S^{2n-1}\cup
S^{2n-1}\times D^{2n}$$ with the boundary of $D^{2n}\times D^{2n}$
given by $\partial(D^{2n}\times D^{2n})=D^{2n}\times S^{2n-1} \cup
S^{2n-1}\times D^{2n}$ so $Z(K;(D^{2n},S^{2n-1}))=S^{4n-1}$.

\item The second example is $(X_1,A_1)=(D^{2},S^{1})$ and $(X_2,A_2)=(D^{3},S^{2})$.
Then $$Z(K;(\underline{X},\underline{A})) = D^{2}\times S^{2}\cup
S^{1}\times D^{3} = S^4.$$

\item The third example exhibits the connection to toric
manifolds where $K'$ is a simplicial complex with $m$ vertices.
Regard $(D^{2n}, S^{2n-1})$ as a pair of subspaces of $\mathbb C^n$
in the standard way with the standard $S^1$-action restricting to an
action on the pair $(D^{2n}, S^{2n-1})$. Thus $T^m$ acts on
$Z(K';(D^{2n}, S^{2n-1}))$. Specialize to $K'=K$. Notice that the
$S^1$ diagonal subgroup $\Delta(T^2)$ acts freely on
$Z(K;(D^{2n},S^{2n-1}))=S^{4n-1}$ where the quotient space
$S^{4n-1}/\Delta(T^m)$ is $\mathbb C \mathbb P^{2n-1}$, complex
projective space of dimension $(2n-1)$. If $n=1$, then $(D^{2n},
S^{2n-1}) = (D^{2}, S^{1})$ and
$$Z(K;(D^2,S^1))=S^3$$ is a case related to toric topology.
That is, $\mathbb C \mathbb P^{1}$ is the toric manifold quotient by
$\Delta(T^2)$ of its associated moment-angle complex.

\item More generally, let $\Delta[m-1]_q$ be the $q$-skeleton of the
$(m-1)$-simplex $\Delta[m-1]$ consisting of all the simplices of
$\Delta[m-1]$ of dimension $\leq q$. Since $$\partial \Delta[m-1]=
\Delta[m-1]_{m-2},$$ there is a homeomorphism
$$Z(\partial\Delta[m-1];(D^{2n},S^{2n-1}))\to S^{2mn-1}$$
and again the diagonal group $\Delta(T^m)=S^1$ acts freely on
$S^{2mn-1}$ with quotient the manifold $\mathbb C \mathbb P^{mn-1}$.
Again when $n=1$, there is a homeomorphism
$$Z(\partial\Delta[m-1];(D^2,S^1))\to S^{2m-1}$$ and the quotient by
the action of the diagonal subgroup is $\mathbb C \mathbb P^{m-1}$,
a description given by $\mathbb C \mathbb P^{m-1}$ as a toric
manifold. A stronger relation between the complex $K'$ and a toric
manifold $M^{2n}$ is given in Section \ref{toric manifolds}.
\end{enumerate}
\end{exm}

\begin{defin}\label{defin:smash.products}
Consider the ordered sequence $I = (i_1, \ldots, i_k)$ with $1 \leq
i_1 <\ldots < i_k \leq m$ together with pointed spaces $Y_1, \ldots,
Y_m$. Then
\begin{enumerate}
\item the length of $I$ is $|I|= k$,
\item the notation $I \leq [m]$ means
$I$ is any subsequence of $(1,\ldots, m)$,
\item $Y^{[m]}=Y_1 \times \ldots \times Y_m,$
\item $Y^{I} = Y_{i_1} \times Y_{i_2} \times \ldots \times
Y_{i_k},$
\item $\widehat{Y}^{I} = Y_{i_1} \wedge \ldots \wedge Y_{i_k},$ and
(redundantly) $\widehat{A}^{I}$ denotes $A_{i_1}\wedge \ldots \wedge
A_{i_k}$ for $I = (i_1, \ldots, i_k)$.
\end{enumerate}

Two more conventions are listed next.
\begin{enumerate}
  \item The symbol $X*Y$ denotes the join of two topological spaces $X$ and
$Y$. If $X$ and $Y$ are pointed spaces of the homotopy type of a
CW-complex, then  $X*Y$ has the homotopy type of the suspension
$\Sigma(X \wedge Y)$.
  \item Given the family of pairs
$(\underline{X},\underline{A})= \{(X_i, A_i)\}^m_{i=1}$ and $I \leq
[m]$, define $$(\underline{X_I},\underline{A_I}) = \{(X_{i_j},
A_{i_j})\}_{j= 1}^{ j=|I|}$$ which is the subfamily of
$(\underline{X},\underline{A})$ determined by $I$.
\end{enumerate}
\end{defin}

Standard constructions for simplicial complexes and associated
posets are recalled next.

\begin{defin}\label{defin:full.subcomplexes.and.realization}
Let $K$ denote a simplicial complex with $m$ vertices.

\begin{enumerate}
\item Recall that the empty simplex $\phi$ is required to be in $K$.

\item Given a sequence  $I = (i_1, \ldots, i_k)$ with $1 \leq i_1 <\ldots <
i_k \leq m $, define $K_I \subseteq K$ to be the {\it full
sub-complex } of $K$ consisting of all simplices of $K$ which have
all of their vertices in $I$, that is $K_I = \{\sigma \cap I |
\sigma \in K\}.$

\item Let $|K|$ denote the geometric realization of the simplicial complex $K$.

\item Associated to a simplicial complex $K$, there is a partially ordered
set (poset) $\bar{K}$ given as follows. A point $\sigma$ in
$\bar{K}$ corresponds to a simplex $\sigma \in K$ with order given
by {\it reverse} inclusion of simplices. Thus $\sigma_1 \leq
\sigma_2$ in $\bar{K}$ if and only if  $\sigma_2 \subseteq \sigma_1$
in $K$. The empty simplex $\phi$ is the unique maximal element of
$\bar{K}$.

Let $P$ be a poset with $p\in P$. There are further posets given by
$$P_{\leq p} = \{q \in P| q \leq p\}$$ as well as
$$P_{< p} = \{q \in P| q < p\}.$$ Thus
$${\bar{K}}_{< \sigma} = \{\tau \in \bar{K}| \tau < \sigma\} =
\{\tau \in K| \tau \supset \sigma\}.$$
\end{enumerate}

\end{defin}

On the other hand, given a poset $P$, there is an associated
simplicial complex $\Delta(P)$ called the  order complex of $P$
which is defined as follows.

\begin{defin}\label{defin:order.complex}
Given a poset $P$, the {\it order complex $\Delta(P)$} is the
simplicial complex with vertices given by the set of points of $P$
and $k$-simplices given by the ordered $(k+1)$-tuples $(p_{1},
p_{2}, \ldots, p_{k+1})$ in $P$ with $p_{1} < p_{2}< \ldots <
p_{k+1}$. It follows that $\Delta(\bar{K}) = \mbox{cone}(K')$ where
$K'$ denotes the barycentric subdivision of $K$.
\end{defin}

The following decomposition is well-known \cite{milnor, james} for
which $\bigvee$ denotes the wedge, and $\Sigma (X)$ denotes the
reduced suspension of $X$.
\begin{thm} \label{thm:T2.1}
Let $(Y_i,y_i)$ be pointed CW-complexes. There is a pointed, natural
homotopy equivalence $$H:\Sigma (Y_1 \times \ldots \times Y_m) \to
\Sigma( \bigvee_{I \leq [m]} \widehat{Y}^I )$$ where $I$ runs over
all the non--empty sub-sequences of $(1,2,\ldots,m)$. Furthermore,
the map $H$ commutes with inductive colimits.
\end{thm}

\begin{remark}\label{remark:naturality.of.Hopf.invariant}
The natural homotopy equivalence $H:\Sigma (Y_1 \times \ldots \times
Y_m) \to \Sigma( \bigvee_{I \leq [m]} \widehat{Y}^I )$ in Theorem
 \ref{thm:T2.1} is defined and used in the proof of Theorem
 \ref{thm:decompositions.for.general.moment.angle.complexes}
below.
\end{remark}

Recall from Definition \ref{defin:smash.products} that
$(\underline{X_I}, \underline{A_I})$ denotes the sub-collection of
$(\underline{X},\underline{A})$ determined by $I$. An application of
Theorem \ref{thm:T2.1} yields the following splitting theorem.
\begin{thm} \label{thm:decompositions.for.general.moment.angle.complexes}
Given $(\underline{X},\underline{A}) =\{(X_i, A_i)\}^m_{i=1}$ where
$(X_i,A_i,x_i)$ are connected, pointed CW-pairs, the homotopy
equivalence of Theorem \ref{thm:T2.1} induces a natural, pointed
homotopy equivalence
$$H: \Sigma(Z(K;(\underline{X},\underline{A})))\to \Sigma(\bigvee_{I \leq [m]}
\widehat{Z}(K_I;(\underline{X_I},\underline{A_I}))).$$
\end{thm}

\begin{remark}\label{remark:non.decomposition}
The spaces $Z(K;(\underline{X},\underline{A}))$ generally do not
decompose as a wedge before suspending. One example is given where
$K$ is the simplicial complex determined by a square, with $4$
vertices and $4$ edges, and where $Z(K;(D^2,S^1))$ is $S^3\times
S^3$.
\end{remark}

The next result is a determination of the homotopy type of the
$Z(K;(\underline{X},\underline{A}))$ in case the inclusions
$A_i\hookrightarrow X_i$ are null-homotopic for every $i\in [m]$.
\begin{thm}\label{T:1.2}
Let $K$ be an abstract simplicial complex and $\overline{K}$ its
associated poset. Let
$(\underline{X},\underline{A})=\{(X_i,A_i,x_i)\}^m_{i=1}$ denote $m$
choices of connected, pointed pairs of $CW$-complexes, with the
inclusion $A_i\subset X_i$ null-homotopic for all $i$. Then there is
a homotopy equivalence
$$\widehat{Z}(K;(\underline{X},\underline{A}))\to\bigvee\limits_{\sigma\in
K} |\Delta(\overline{K}_{<\sigma})|*\widehat{D}(\sigma).$$
\end{thm}

This theorem is a generalization of the wedge lemma of
Welker-Ziegler-\v{Z}ivaljevi\'c in \cite{welker.ziegler.zivaljevic},
applied to generalized moment-angle complexes. Substituting Theorem
\ref{T:1.2} in Theorem
\ref{thm:decompositions.for.general.moment.angle.complexes} yields
the next result.

\begin{thm}\label{T:1.3}
Let $K$ be an abstract simplicial complex with $m$ vertices, and let
$$(\underline{X},\underline{A})=\{(X_i,A_i,x_i)\}^m_{i=1}$$ denote $m$
choices of connected, pointed pairs of $CW$-complexes with the
inclusion $A_i\subset X_i$ null-homotopic for all $i$. Then there is
a homotopy equivalence
$$\Sigma (Z(K;(\underline{X},\underline{A}))) \to \Sigma(\bigvee\limits_I(\bigvee\limits_{\sigma\in
K_I}|\Delta((\overline{K}_I)_{<\sigma}|*\widehat{D}(\sigma))).$$
\end{thm}

Two special cases of Theorem \ref{T:1.3} are presented next where
either (i) the $A_i$ are contractible for all $i$ or (ii) the $X_i$
are contractible for all $i$. In the case for which all $A_i$ are
contractible, the decomposition given in Theorem
\ref{thm:decompositions.for.general.moment.angle.complexes} has the
property that $\widehat{X}^\s$ is contractible for all $\s \neq
[m]$. The one remaining summand in Theorem
\ref{thm:decompositions.for.general.moment.angle.complexes} occurs
for $\s = [m]$ and thus $\widehat{X}^\s = \widehat{X}^{[m]}$ as
given in Definition \ref{defin:smash.products}. The next result
follows at once.

\begin{thm} \label{thm:contractible.A}
If all of the $A_i$ are contractible with $X_i$ and $A_i$ closed
CW-complexes for all $i$, then there is a homotopy equivalence
\[
\widehat{Z}(K;(\underline{X}, \underline{A})) =
\begin{cases}
* & \text{ if $K$ is not the simplex $\Delta[m-1]$, and}\\
\widehat{X}^{[m]} & \text{if $K$ is the simplex $\Delta[m-1]$.}
\end{cases}
\]
\end{thm}

Notice that the spaces
$\widehat{Z}(K_I;(\underline{X_I},\underline{A_I}))$ of Theorem
\ref{thm:decompositions.for.general.moment.angle.complexes} are all
contractible unless $K_I$ is a simplex of $K$ by Theorem
\ref{thm:contractible.A}. The simplices of $K$ have been identified
with certain increasing sequences $I \leq [m]$. Thus, the next
result follows immediately from Theorems
\ref{thm:decompositions.for.general.moment.angle.complexes} and
\ref{thm:contractible.A}.

\begin{thm} \label{thm:more.contractible.A}
If  all of the $A_i$ are contractible with $X_i$ and $A_i$ closed
CW-complexes for all $i$, then there is a homotopy equivalence
$$\Sigma( Z(K;(\underline{X_I},\underline{A_I}))\to  \Sigma( \bigvee_{I \in K} \widehat{X}^{I}).$$
\end{thm}

In the situation where $(X_i,A_i,x_i)=(X,A,x_0)$ for all $i$, the
next result follows at once.
\begin{cor}\label{cor:same.f.vectors}
If $X_i = X$ and all of the $A_i$ are contractible, with $K$ and
$K'$ simplicial complexes both having $m$ vertices as well as the
same number of simplices in every dimension, then there is a
homotopy equivalence
$$\Sigma(Z(K;(X,A))) \to \Sigma(Z(K';(X,A))).$$
\end{cor}

\begin{remark}\label{remark:f.vectors}
The condition in Corollary \ref{cor:same.f.vectors} that two
simplicial complexes $K$ and $K'$ both have $m$ vertices as well as
the same number of simplices in every dimension is the definition of
having {\it the same $f$-vectors}.
\end{remark}

Let $\overline{P}(Y)$ denote the reduced Poincar\'e series for a
finite, connected CW-complex.
\begin{cor}\label{cor:poincare}
If $(X,A)$ is a pair of finite CW-complexes with $A$ contractible,
and $|K|$ is connected, then
$$\overline{P} (Z(K;(X,A)) = \sum^n_{k = 0}f_k(\overline{P}(X))^{k+1}$$
where $n$ is the dimension of $K$, and $f_k$ is the number of
$k$-simplices of $K$.
\end{cor}

When specialized to $(\underline{X},\underline{A})$ for which the
spaces $X_i$ are contractible, there is a precise identification of
$\bigvee_{I \leq [m]}
\widehat{Z}(K_I;(\underline{X_I},\underline{A_I}))$ obtained by
appealing to the ``Wedge Lemma" in work of
Welker-Ziegler-\v{Z}ivaljevi\'c in \cite{welker.ziegler.zivaljevic}
on homotopy colimits of spaces. The next theorem addresses the case
for which all of the $X_{i}$ are contractible. Recall the join of
two spaces $X*Y$ as well as the notation
$\widehat{A}^{[m]}=A_1\wedge\ldots\wedge A_m$ from Definition
\ref{defin:smash.products}.
\begin{thm}\label{T2.19}
Let $K$ be an abstract simplicial complex with $m$ vertices and let
$$(\underline{X},\underline{A})=\{(X_i,A_i,x_i)\}^m_{i=1}$$ be a family
of connected, pointed CW-pairs. If all of the $X_i$ are
contractible, then there are homotopy equivalences
$$\widehat{Z}(K;(\underline{X},\underline{A}))\to |K|\ast\widehat{A}^{[m]}
\to \Sigma (|K|\wedge\widehat{A}^{[m]}).$$
\end{thm}

Since all of the $X_i$ are contractible, all of the
$\widehat{D}(\sigma)$ are contractible with the possible exception
of $\widehat{D}(\emptyset)=A_1\wedge\ldots\wedge A_m$. Since the
element $\emptyset$ is the maximal element under reverse inclusion
in $\bar{K}$, the order complex $\Delta(K_{<\emptyset})$ is equal to
the barycentric subdivision of $K$ denoted $SdK$. Furthermore, there
are homeomorphisms on the level of geometric realizations
$|\Delta(\overline{K})|\to |SdK|\to |K|$.

\begin{remark}\label{remark:david.stone}
For the case $(X,A) = (D^2,S^1)$ and $K$ is a simplicial complex
with $m$ vertices, David Stone \cite{stone} has shown that the
homotopy equivalence $$\widehat{Z}(K;(D^2, S^1)) \to
\Sigma^{m+1}|K|$$ given by Theorem \ref{T2.19}  extends to a
homeomorphism. He does this by ``linearizing" a smash product of
discs using joins. This result may be interpreted as giving, for any
simplicial complex K on $m$-vertices, a model of $\Sigma^{m+1}|K|$
inside an $m$-fold smash product of two-discs which preserves the
combinatorial structure of $K$ in a natural way.
\end{remark}

The next result is now a consequence of Theorems
\ref{thm:decompositions.for.general.moment.angle.complexes} and
\ref{T2.19}.

\begin{thm} \label{thm:null.homotopy.for.A.to.X}
If all of the $X_i$ in $(\underline{X},\underline{A})$ are
contractible with $X_i$ and $A_i$ closed CW-complexes for all $i$,
there is a homotopy equivalence $$\Sigma Z(K;
(\underline{X},\underline{A})) \to \Sigma (\bigvee_{I \notin K}
|K_I|*\widehat{A}^I).$$
\end{thm}

\begin{remark}\label{remark:remark}
An important case of Theorem \ref{thm:null.homotopy.for.A.to.X} is
given by setting $X_i$ equal to the cone on $A_i$.
\end{remark}

Since $D^2$ is contractible, the next corollary follows at once from
Theorem \ref{thm:null.homotopy.for.A.to.X}.
\begin{cor}\label{cor:C2.5}
Let $(X_i, A_i, x_i)$ denote the triple $(D^2,S^1, *)$ for all $i$.
Then there are homotopy equivalences
$$\Sigma (Z(K;(D^2,S^1)))\to  \Sigma (\bigvee_{I \notin K}
|K_I|*S^{|I|})\to \bigvee_{I \notin K}\Sigma^{2 + |I|}|K_I|.$$
\end{cor}

Notice that the last corollary implies a decomposition for the
homology of $Z(K;(D^2,S^1))$ for any homology theory. The associated
additive decomposition for singular homology is due to Hochster
\cite{hochster}, Goresky-MacPherson \cite{goresky.macpherson},
Buchstaber-Panov \cite{buchstaber.panov}, Jewell \cite{jewell}, and
others. An analogous consequence is given next.

\begin{cor}\label{cor:C2.7}
Let $(X_i, A_i, x_i)$ denote the triple $(D^{n+1},S^n, *)$ for all
$i$. Then there are homotopy equivalences
$$\Sigma (Z(K;(D^{n+1},S^n)))\to
\Sigma (\bigvee_{I \notin K} |K_I|*S^{n|I|}) \to \bigvee_{I \notin
K} \Sigma^{2 + n|I|}|K_I|.$$
\end{cor}

The cohomology ring structure is given by Panov \cite{panov}, a
feature that does not follow from the above corollary.

\begin{remark}\label{remark:arrangements.macs}
A simplicial complex $K$ determines the complement of a complex
coordinate subspace arrangement as outlined in
\cite{buchstaber.panov}. This space is $Z(K;(\C, \C^{*}))$ which has
$Z(K;(D^{2},S^1))$ as a strong deformation retract
\cite{buchstaber.panov}. Similarly, $Z(K;(\R^{n+1}, \R^{n+1} -
\{{0}\}))$ is the complement of a certain arrangement of real
subspaces having $Z(K;(D^{n+1},S^n))$ as a strong deformation
retract.
\end{remark}

Shifted simplicial complexes are defined next.
\begin{defin}\label{defin:shifted}
A simplicial complex on $m$ vertices is {\it shifted} if there
exists a labeling of the vertices by $1$ through $m$ such that for
any face, replacing any vertex of that face with a vertex of smaller
label and not in that face results in a collection which is also a
face. For a shifted complex, the geometric realization of $K$ and
every $K_I$ is homotopy equivalent to a wedge of spheres.
\end{defin}

The next corollary follows.
\begin{cor}\label{cor:C2.8}
If $K$ is a shifted complex, then $\Sigma Z(K;(D^{n+1}, S^n))$ is
homotopy equivalent to a wedge of spheres.
\end{cor}

\begin{remark}\label{remark:grbic.theriault}
In the special case for which $K$ is a shifted complex, a stronger
result than that of Corollary \ref{cor:C2.8} was proven by
Grbic-Theriault \cite{grbic.theriault}. They prove that $Z(K;(D^{2},
S^1))$ is homotopy equivalent to a wedge of spheres without
suspending.
\end{remark}

The theorems of Porter and Grbic-Theriault  support the following.
\begin{conjecture}\label{conjecture:conjecture}
If $K$ is a shifted complex, then
$$Z(K;(\underline{CX},\underline{X}))\quad{\rm and}\quad\bigvee_{I\notin K}|K_I| *\widehat{X}^I$$
are of the same homotopy type.
\end{conjecture}

Denham and Suciu prove an elegant Lemma numbered $2.9$ in
\cite{denham.suciu} which relates fibrations and moment-angle
complexes. That Lemma is stated next.
\begin{lem} \label{denham.suciu.fibrations}
Let $p: (E,E') \to (B,B')$ be a map of pairs, such that both $p: E
\to B$ and $p|E' : E' \to B'$ are fibrations, with fibres $F$ and
$F'$, respectively. Suppose that either $F = F'$ or $B = B'$. Then
the product fibration, $p^n : E^n \to  B^n$, restricts to a
fibration  $Z(K;(F,F^{\prime})) \to Z(K;(E,E^{\prime})) \to
Z(K;(B,B^{\prime}))$. Moreover, if $(F, F')\to (E,E') \to (B,B')$ is
a relative bundle (with structure group G),
 then the above bundle has structure
group $G^n$.
\end{lem}

Notice that Lemma \ref{denham.suciu.fibrations} extends easily to
generalized moment-angle complexes as follows (with details of proof
omitted).

\begin{lem}\label{porter}
Let $p_i:(E_i,E'_i)\to(B_i,B_i')$ be a map of pairs, such that
$p:E_i\to B_i$ and $p|E'_i:E'_i\to B'_i$ are fibrations with fibers
$F_i$ and $F'_i$ respectively. Let $K$ be a simplicial complex with
$m$ vertices.
\begin{enumerate}
  \item If $B_i=B_i'$ for all $i$ then,
$Z(K;(\underline{B},\underline{B}'))=B_1\times\ldots\times B_m$.
  \item If $F_i=F'_i$ for all $i$, then
$Z(K;(\underline{F},\underline{F}'))=F_1\times\ldots\times F_m$.
  \item In any one of these two cases, the following is a fibration:
$$Z(K;(\underline{F},\underline{F}'))\to Z(K;(\underline{E},\underline{E}'))\to
Z(K;(\underline{B},\underline{B}').$$
\end{enumerate}
\end{lem}

Natural consequences of Lemma (\ref{porter}) arose earlier in the
work of G.~Porter \cite{porter}, T. Ganea \cite{ganea} and A.~Kurosh
\cite{kurosh} as discussed next. Porter considered the cases
$\{X_i,*_i,*_i)\}^m_{i=1}$. In this case, consider the natural
filtration of the product $X_1\times\ldots \times X_m$ as defined by
G.~W.~Whitehead with $j^{th}$-filtration given by the space
$$W_j(\prod_{1 \leq i\leq m}X_i)=\{(y_1,\ldots,y_m)|\ y_i=*_i,
\hbox{ the base-point of} \ X_i \ \hbox{for at least}\ m-j\
\hbox{values of}\ i\}.$$ The subspace $W_j(\prod_{1 \leq i \leq m}
X_i)$ of the product is precisely a choice of a generalized
moment-angle complex as follows.

Let $\Delta[m-1]_q$ denote the $q$-skeleton of $\Delta[m-1]$, that
is $$\Delta[m-1]_q=\{I\subset [m]\biggl| |I|\leq q+1\}.\biggr.$$
Then
$$W_j(\prod X_i)=Z(\Delta[m-1]_{j-1};(\underline{X},\underline{*})).$$
In this case, Porter's result gives an identification of the the
homotopy theoretic fibre of the inclusion $$W_j(\prod_{1 \leq i \leq
m} X_i) \subset \prod_{1 \leq i \leq m} X_i$$ a result which follows
directly from the classical path-loop fibration over $X$ and Lemma
\ref{porter}, features elucidated next.

Two standard constructions are used here. Recall that the path-space
$PX$ of a pointed space $(X,*)$ is the space of pointed continuous
maps $f:[0,1] \to X$ with $f(0) = *$, the base-point of $X$. The
evaluation map
$$e_1: PX \to X$$ is defined by $e_1(f)= f(1)$ and is a fibration
with fibre $\Omega X$ the space of continuous maps $f:[0,1] \to X$
with $f(0) = * = f(1)$. Observe that there is a fibration of pairs

\[
\begin{CD}
(\Omega X,\Omega X)  @>{}>> (PX, \Omega X) @>{ e_1}>> (X,*).
\end{CD}
\]

Next consider the (unreduced) cone $CY$ over a space $Y$ defined as
$$C(Y) = [0,1] \times Y/ \approx,$$ the quotient obtained from the equivalence
relation $(1,a) \approx (1,b)$ for all $a,b \in Y$ with equivalence
classes denoted $[t,y]$. Observe that there is a map
$$\kappa: X \to X \times C(PX)$$ with $$\kappa(x) = (x,[0,f_x])$$ where
$f_x: [0,1] \to X$ with $f_x(t) = x$. The map $\kappa$ is evidently
a homotopy equivalence.

Consider the pair $$ (X \times C(PX),PX)$$ for which $PX$ is the
subspace of $X \times C(PX)$ given by pairs $(f(1), [0,f])\in X
\times C(PX)$. Observe that there is a fibration of pairs
\[
\begin{CD}
(C(PX),\Omega X)  @>{}>> (X \times C(PX),PX) @>{\pi_X \times e_1}>>
(X,X)
\end{CD}
\] for which $\pi_X:X \times C(PX) \to X$ is the natural projection.
Also, observe that there is an equivalences of pairs $$(X,*) \to (X
\times C(PX),PX).$$ Thus by Lemma \ref{porter}, there is a fibration
\[
\begin{CD}
Z(K;(\underline{PX},\underline{\Omega X}))  @>{}>> Z(K;(\underline{X
\times PX},\underline{PX})) @>{\pi_X \times e_1}>>
Z(K;(\underline{X},\underline{X}))
\end{CD}
\] as well as the associated fibration (up to homotopy)
\[
\begin{CD}
Z(K;(\underline{PX},\underline{\Omega X}))  @>{}>>
Z(K;(\underline{X},\underline{*})) @>{\pi_X \times e_1}>>
Z(K;(\underline{X},\underline{X}))
\end{CD}
\] a remark recorded as the next Corollary.

\begin{cor} \label{cor:proof.of.porter}
If all of the $X_i$ are path-connected, the homotopy theoretic fibre
of the inclusion $ Z(K;(\underline{X},\underline{*})) \subset
Z(K;(\underline{X},\underline{X}))$ is
$Z(K;(\underline{PX},\underline{\Omega X})).$

Thus if all of the $X_i$ are path-connected, the homotopy theoretic
fibre of the inclusion $W_j(\prod_{1 \leq i \leq m} X_i) \subset
\prod_{1 \leq i \leq
m}X_i=Z(\Delta[m-1]_{j-1};(\underline{X},\underline{*}))$ is
$$Z(\Delta[m-1]_{j-1};(\underline{PX},\underline{\Omega X})).$$
\end{cor}

In addition, a second Theorem of Porter \cite{porter} gives the
structure of the fibre
$Z(\Delta[m-1]_{j-1};(\underline{PX},\underline{\Omega X}))$. The
next Theorem provides a decomposition of the suspension of
$Z(\Delta[m-1]_{q-1};(\underline{CY},\underline{Y}))$ where $CY$
denotes the cone over a connected CW-complex $Y$. Let
$(\underline{CY},\underline{Y})=\{(CY_i,Y_i,x_i)\}^m_{i=1}$ be a
family of connected, pointed CW-pairs and recall that
$\widehat{Y}^I$ is $Y_{i_1}\wedge\ldots\wedge X_{i_k}$ if
$I=(i_1,\ldots,i_k)$. The next result follows from Theorem
\ref{T2.19}.

\begin{thm}\label{T:1.18}
Let $(\underline{CY},\underline{Y})=\{(CY_i,Y_i,x_i)\}^m_{i=1}$ be a
family of connected, pointed CW-pairs. Then there is a homotopy
equivalence
$$ \Sigma(Z(\Delta[m-1]_{q};(\underline{CY},\underline{Y})))
\to\Sigma(\bigvee_{|I|>q+1}\bigvee_{t_I}(\Sigma
^{|I|+1}\widehat{Y}^I))$$ where $t_I$ is the binomial coefficient
$\binom {|I|+1} {q+1}$.
\end{thm}

When $q=m-2$, $Z(\Delta[m-1]_{m-2};(\underline{X};\underline{*}))$
is the fat wedge $W_{m-1}(\Pi X_i)$ and the homotopy fiber is
$Z(\Delta[m-1]_{m-2};(\underline{PX},\underline{\Omega X}))$. In
this case, it follows that there is a homotopy equivalence
$$\Sigma(Z(K_{m-1,m-2};(\underline{PX},\underline{\Omega X}))) \to \Sigma(\Omega X_1 *\ldots *\Omega X_m).$$

T.~Ganea \cite{ganea} had identified the homotopy theoretic fibre of
the natural inclusion $$X \bigvee Y \to X \times Y$$ as the join
$\Omega(X)*\Omega(Y)$ in case $X$ and $Y$ are path-connected,
pointed spaces of the homotopy type of a CW-complex. This example is
a special case of the homotopy theoretic fibre for the inclusion of
a generalized moment-angle complex in a product corresponding to the
simplicial complex $K$ given by two disjoint points.

The next result gives a determination of the cohomology ring of
$Z(K;(\underline{X},\underline{A}))$ under the conditions that each
$A_i$ is contractible, and coefficients are taken in a ring $R$
where
\begin{enumerate}
  \item $R$ is a field, or
  \item the cohomology of $X$ satisfies the strong form of the K\"unneth
  Theorem with coefficients in $R$.
\end{enumerate} Examples in the case $R = \mathbb Z$ are given by spaces $X$
of finite type which have torsion free cohomology over $\mathbb Z$.

Notice that there is a natural inclusion
$j:Z(K;(\underline{X},\underline{A}))\subset
\prod\limits^m_{i=1}X_i$. Theorem \ref{thm:contractible.A} implies
that $$j^*:H^*(\prod X_i;R) \to
H^*(Z(K;(\underline{X},\underline{A}));R)$$ 
is onto and that the kernel of $\theta^*$ is defined to be {\it the
generalized Stanley-Reisner ideal} $I(K)$ which is generated by all
elements $x_{j_1}\otimes x_{j_2}\otimes\ldots \otimes x_{j_l}$ for
which $x_{j_t}\in \bar{H}^*(X_{j_t};R)$ and the sequence
$J=(j_1,\ldots, j_l)$ is not a simplex of $K$. This construction
provides a useful extension of the Stanley-Reisner ring
\cite{davis.jan} as stated next.

\begin{thm}\label{thm:cohomology.for.contractible.A}
Let $K$ be an abstract simplicial complex with $m$ vertices and let
$$(\underline{X},\underline{A})=\{(X_i,A_i,x_i)\}^m_{i=1}$$ be $m$
pointed, connected CW-pairs. If all of the $A_i$ are contractible
and coefficients are taken in a ring $R$ for which either
\begin{enumerate}
\item $R$ is a field, or
  \item the cohomology of $X$ with coefficients in $R$ satisfies the strong form
of the K\"unneth Theorem,
\end{enumerate} then there is an isomorphism of algebras
$$(\bigotimes^m_{i=1}H^*(X_i;R)/I(K)\to H^*(Z(K;(\underline{X},\underline{A}));R).$$
Furthermore, there are isomorphisms of underlying abelian groups
given by
$$E^*(Z(K;(\underline{X},\underline{A}))\to
\bigoplus\limits_{\sigma\in K}E^*(\widehat{D}(\sigma))$$ for any
reduced cohomology theory $E^*$.
\end{thm}

\begin{remark}\label{remark:any.theory}
An analogous result for the cohomology ring structure is satisfied
for any cohomology theory $E^*(-)$ for which $E^*(X)$ satisfies the
strong form of the K\"unneth Theorem \cite{astey}.
\end{remark}

Denham and Suciu \cite{denham.suciu} consider an action of a
topological group $G$ on the pair $(X,A)$ together with the natural
induced action of $G^m$ on $Z(K;(X,A))$ where $K$ has $m$ vertices.
The associated Borel construction (homotopy orbit space) is
$$EG^m \times_{G^m}Z(K;(X,A)).$$ The applications here will be
restricted to topological groups $G$ which are CW-complexes. Regard
$G$ as a subspace of $EG$ given by the orbit of a point
corresponding to the identity in $G$. Denham-Suciu proved the
following using their Lemma 2.10 \cite{denham.suciu}).

\begin{thm} \label{thm:denham.suciu.BG}
The space $Z(K;(BG,*))$ is homotopy equivalent to the homotopy orbit
space $$EG^m \times_{G^m}Z(K;(EG,G)).$$

\noindent Furthermore, the natural projection $$EG^m
\times_{G^m}Z(K;(EG,G)) \to BG^m$$ is a fibration with fibre
$$Z(K;(EG,G)).$$
\end{thm}

Recall the $T^m$-action on  $Z(K;(D^2, S^1))$ induced by the natural
$S^1$-action on the pair $(D^2, S^1)$.
\begin{defin}\label{defin:DJ.spaces}
Define the Davis-Januszkiewicz space
$$\mathcal D\mathcal J(K) = ET^m \times_{T^m}Z(K;(D^2,S^1))= Z(K;(BS^1,\ast)).$$
\end{defin}

The space $\mathcal D\mathcal J(K)$ is an example of the
construction addressed above by Denham-Suciu \cite{denham.suciu}.
The cohomology ring of $\mathcal D\mathcal J(K)$ was computed first
in work of Davis-Januszkiewicz \cite{davis.jan} as well as work of
Buchstaber-Panov \cite{buchstaber.panov}. An elegant proof was given
in work of Denham-Suciu \cite{denham.suciu}.

To state their results, let $I= (i_1,\ldots, i_k) $ denote a simplex
in $\Delta[m-1]$. Then, let $x_{I}= x_{i_1}\ldots x_{i_k}$ for which
the $x_i$, $1 \leq i \leq m$, are the natural choice of algebra
generators for the integral cohomology ring of $ BT^m$.

\begin{thm} \label{thm:denham.suciu.cohomology}
The integral cohomology ring of the Davis-Januszkiewicz space
$$\mathcal D\mathcal J(K) = Z(K;(BS^1,\ast))$$ is isomorphic as an algebra
to $$\mathbb Z[x_1, \ldots,x_m]/I(K)$$ where $I(K)$ denotes the
ideal generated by $x_{I}= x_{i_1}\ldots x_{i_k}$ for which $I=
(i_1,\ldots, i_k)$ is not a simplex in $K$. In particular, $\mathbb
Z[x_1, \ldots,x_m]/I(K)$ is the Stanley-Reisner ring of $K$.
\end{thm}

\begin{remark}\label{remark:cp.infinity}
Theorem \ref{thm:denham.suciu.cohomology} is a special case of
Theorem \ref{thm:cohomology.for.contractible.A} where $(X_i,x_i) =
(BS^1, *)$.
\end{remark}

The decompositions for the suspensions of generalized moment-angle
complexes given in Theorem
\ref{thm:decompositions.for.general.moment.angle.complexes}
specialize to decompositions for $Z(K;(BG,\ast))$ as follows.
\begin{cor} \label{cor:split.toric.variety}
There is a homotopy equivalence
$$ \Sigma(Z(K;(BG,\ast)))\to \Sigma (\bigvee_{I \leq [m]}
\widehat{Z}(K_I; (BG,\ast))).$$ Thus in case $G = T = S^1,$ there is
an additive decompostion of Stanley-Reisner ring given by the
cohomology algebra of $\mathcal D\mathcal J(K) = ET^m
\times_{T^m}Z(K;(D^2,S^1))$ together with an isomorphism of
underlying abelian groups $$H^i(\mathcal D\mathcal J(K);\mathbb Z)
\to H^i(\bigvee_{I \in K} \widehat{(BS^1)^I}; \mathbb Z).$$
Furthermore, there is a decomposition
$$E_*(\mathcal D\mathcal J(K);\mathbb Z) \cong
E_*(\bigvee_{I \in K}\widehat{(BS^1)^I})$$ for any homology theory
$E_*$.
\end{cor}

\begin{remark}\label{remark:bjorner.sarkaria}
An analogous algebraic decomposition was given by Bj\"orner and
Sarkaria in \cite{bjorner.sarkaria} for the special case of
$(D^2,S^1)$ and singular cohomology with integer coefficients. The
decomposition in Theorem \ref{thm:cohomology.for.contractible.A} is
also closed with respect to the Steenrod operations in mod-$p$
cohomology arising from the geometric splitting of the suspension.
It is not apparent that the decomposition of Bj\"orner and Sarkaria
after reduction mod-$p$ preserves the action of the mod-$p$ Steenrod
algebra.
\end{remark}

Some features of generalized moment-angle complexes together with
their stable decompositions in Theorem
\ref{thm:decompositions.for.general.moment.angle.complexes} extend
directly to simplicial spaces $X_*$ and their geometric realizations
denoted $|X_*|$. To do so, some notation is given next by analogy
with Definition \ref{defin:moment.angle.complex}.

\begin{defin}\label{defin:simplicial.macs}
Let $(\underline{X}_*,\underline{A}_*)$ denote the collection of
pairs of simplicial spaces $$\{(X_*(i), A_*(i))\}^m_{i=1}$$ and let
$(\underline{X}_n,\underline{A}_n)$ denote the collection of pairs
$\{(X_n(i), A_n(i))\}^m_{i=1}$ where $X_n(i)$ denotes the $n$-th
space in $X_*(i)$. Define
$$Z(K;(\underline{X}_*,\underline{A}_*))_n = Z(K;(\underline{X}_n,\underline{A}_n)).$$
\end{defin}

The next lemma gives the property that the collection of spaces
$Z(K;(\underline{X}_*,\underline{A}_*))_n $ for all $n \geq 0$ is
naturally a simplicial space which is denoted by
$Z(K;(\underline{X}_*,\underline{A}_*))$.
\begin{lem} \label{lem:simplicial.moment.angle}
The natural inclusions
$$i_n:Z(K;(\underline{X}_*,\underline{A}_*))_n =Z(K;(\underline{X}_n,\underline{A}_n)) \to X_n(1) \times \ldots \times X_n(m)$$
are closed with respect to the face and degeneracy operations in the
product $$X_*(1) \times \ldots \times X_*(m).$$ That is the
following properties are satisfied.
\begin{enumerate}
\item $d_i: Z(K; (\underline{X}_*,\underline{A}_*)_n \to Z(K; (\underline{X}_*,\underline{A}_*))_{n-1}$ and
\item $s_j: Z(K; (\underline{X}_*,\underline{A}_*)_n \to Z(K; (\underline{X}_*,\underline{A}_*))_{n+1}.$
\end{enumerate}

\noindent Hence
\begin{enumerate}
\item $s_j(x) \in A_{n+1}(i)$ if and only if $x \in A_{n}(i)$, and
\item the natural inclusion gives a map of simplicial space
$$i_*:Z(K; (\underline{X}_*,\underline{A}_*)) \to X_*(1) \times \ldots \times X_*(m).$$
\end{enumerate}
\end{lem}

The geometric realization of a simplicial space $X_*$ is recalled
next for convenience of the reader.
\begin{defin}\label{defin:geometric.realization}
The geometric realization $|X_*|$ is the quotient
$$ \bigcup_{0 \leq i}\Delta[i] \times X_i/(\sim)$$ where
$\Delta[i]$ denotes the $i$-simplex and $\sim$ is the equivalence
relation generated by

\begin{enumerate}
\item $(u,d_ix)\sim (\partial_i u,x)$ for $u \in \Delta[r-1]$
and $x \in X_r$, and
\item $(\sigma_j(v),y)\sim (v,s_j(x))$ for $v \in \Delta[r+1]$
and $y \in X_r$.
\end{enumerate} Let $[u,x]$ in $|X_*|$ denote the equivalence class
of the pair $(u,x)$ in $\Delta[r] \times X_r$.
\end{defin}

Given the collection of pairs of simplicial spaces
$$(\underline{X}_*,\underline{A}_*) = \{(X_*(i), A_*(i))\}^m_{i=1},$$
let $(\underline{|X_*|},\underline{|A_*|})$ denote the collection of
spaces $\{(|X_*(i)|,|A_*(i)|)\}^m_{i=1}$ where $|X_*(i)|$ and
$|A_*(i)|$ denote the respective geometric realizations.

It is a classical fact \cite{milnor3} that there are homeomorphisms
$$\pi: |X_*(1)\times \ldots \times X_*(m)| \to |X_*(1)|\times \ldots
\times |X_*(m)|$$ for simplicial spaces $X_*(i)$ for $ 1 \leq i \leq
m$ in case either of the following are satisfied.
\begin{enumerate}
  \item Each $X_*(i)$ is countable, or
  \item the spaces $|X_*(i)|$ are locally finite.
\end{enumerate} A variation in the underlying point-set topology is given in
\cite{may} where it suffices to assume that all spaces are compactly
generated and weak Hausdorff.

The homeomorphism $\pi$ is induced by the product of the natural
projection maps $$\pi_j:X_*(1)\times \ldots \times X_*(m) \to
X_*(j)$$ which gives a map $$|\pi_j|:|X_*(1)\times \ldots \times
X_*(m)| \to |X_*(j)|.$$ The inverse of $\pi = |\pi_1|  \times \ldots
\times |\pi_m|$ $$\eta:\prod_{1 \leq i \leq m} |X_*(i)| \to |
\prod_{1 \leq i \leq m} X_*(i)|$$ is induced by sending the class of
$\prod_{1 \leq i \leq m} |u_i,x_i|$ to $|w,\prod_{1 \leq i \leq m}
s_{J_i}(x_i)|$ where $w$ and the possibly iterated degeneracies
$s_{J_i}(x_i)$ are specified in \cite{milnor3}.

Since the natural inclusion $i_n:Z(K;(X_n,A_n)) \to X_n(1) \times
\ldots \times X_n(m)$ is a monomorphism of simplicial spaces by
Lemma \ref{lem:simplicial.moment.angle}, the realization
$$|Z(K; (\underline{X}_*,\underline{A}_*))|$$ is a subspace of $|X_*(1) \times \ldots \times  X_*(m)|$.
Furthermore, let $$i:Z(K;(|\underline{X}|_*,|\underline{A}|_*)) \to
|X_*(1)| \times \ldots \times |X_*(m)|.$$ denote the natural
inclusion. The following statement is proven by repeating a proof in
\cite{milnor3}.
\begin{thm} \label{thm:simplicial.moment.angle.15.march.2007}
There is a commutative diagram
\[
\begin{CD}
 Z(K;(|\underline{X}|_*,|\underline{A}|_*)) @>{h}>>   |Z(K; (\underline{X}_*,\underline{A}_*))|\\
@VV{i}V          @VV{|i_*|}V          \\
|X_*(1)|\times \ldots \times |X_*(m)| @>{\eta}>> |X_*(1) \times
\ldots \times X_*(m)|
\end{CD}
\] where $h$ is the restriction of $\eta$, $$h =
\eta|_{|Z(K;\underline{X_*})|}.$$

Furthermore, if either
\begin{enumerate}
  \item the simplicial spaces $X_*(i)$ are countable,
  \item the spaces $|X_*(i)|$ are locally finite, or
  \item the spaces $X_n(i)$ are assumed to be compactly generated and weak
  Hausdorff,
\end{enumerate}
then the map $$h:Z(K;(|\underline{X}|_*,|\underline{A}|_*)) \to
|Z(K; (\underline{X}_*,\underline{A}_*))|
$$ is a homeomorphism.
\end{thm}

\section{Toric manifolds}\label{toric manifolds}

This section is a brief introduction to properties of toric
manifolds with an emphasis on relationships with moment-angle
complexes $Z(K;(D^2,S^1))$. Following Davis-Januszkiewicz
\cite{davis.jan} and Buchstaber-Panov \cite{panov} define a toric
manifold starting with $T^n$ the real $n$-torus
$$T^n=S^1\times\ldots \times S^1\qquad (n \quad\hbox{\rm factors}).$$
Notice that $T^n$ acts on $\mathbb C^n$, $n$-space by:
$$(t_1,\ldots,t_n)\cdot (z_1,\ldots,z_n)=(t_1z_1,t_2z_2,\ldots,t_nz_n)$$
where $t_iz_i$ is the natural multiplication of two complex numbers.
The quotient space $\mathbb C^n/T^n$ is homeomorphic to
$$\mathbb R^n_+=\{(x_1,\ldots,x_n)| \ x_i\in R, x_i\geq 0\}.$$

Let $M^{2n}$ be a $2n$-manifold. Say that $T^n$ acts on $M^{2n}$
with {\it locally standard action} if (i) every point $x\in M^{2n}$
has a coordinate neighborhood $U_x$ invariant under the action of
$T^n$, (ii) there is a homeomorphism $\varphi_x:U_x\to V_x\subset
\mathbb R^{2n}=\mathbb C^n$ with $V_x$ open, and (iii) for every
$t\in T^n$ and $y\in U_x$, the following are equal
$$\varphi_x(t\cdot y)=\alpha_x(t)\cdot\varphi_x(y)$$
where $\alpha_x$ is an automorphism of $T^n$ with the standard
action on the right-hand side. It follows that $M^{2n}/T^n$ is the
union of spaces $\mathbb R^n_+$. Say that $M^{2n}$ is {\it a toric
manifold over $P^n$}, if the orbit space is homeomorphic to a simple
$n$-polytope $P^n$.

Recall that $P^n$ can be described as the intersection of $m$
half-spaces. Notice that $P^n$ has faces of codimension $1$ to $n$.
Assume that there are $m$ codimension-$1$ faces, $F_1,\ldots, F_m$
and if the intersection $F_{i_0}\cap\ldots\cap F_{i_k} \neq
\emptyset$ for $1 \leq i_0 < i_1 < \ldots < i_k$, it is a
codimension $k$-face.

Let $\sigma=(i_0,\ldots,i_k)$ denote an increasing sequence in
$[m]$. Consider the intersection defined by $F_\sigma =
F_{i_0}\cap\ldots\cap F_{i_k}$ which will be assumed to be
non-empty. Then the set of all such $\sigma$ for which $F_\sigma $
is not empty defines a simplicial complex $K$ ( which contains the
empty set by definition ). If $K$ contains an element $\sigma$ which
is not the empty sequence, then $K$ is the dual complex of the
boundary complex of the simple convex polytope $P^n$. Moreover,
$|K|=S^{n-1}$. Note that this particular $K$ is denoted $K_P$ in
\cite{davis.jan}, and \cite{buchstaber.panov.2}.

Notice that $T^n$ acts on $M^{2n}$. Let $\pi:M^{2n}\to P^n$ denote
the quotient map. Observe that $T^n$ acts on $\pi^{-1}({\rm
Int}(F_\sigma))$ with constant isotropy subgroup, denoted
$\lambda(F_\sigma)$, a $(k+1)$-subtorus of $T^n$, where
$k+1=|\sigma|=$ length of $\sigma$ and ${\rm Int}(F_\sigma)$ denotes
the interior of $F_\sigma$.

In particular $T^n$ acts freely on $\pi^{-1}({\rm Int}(P^n))$.
However, $T^n$ acting on $M^{2n}$ has fixed points which are the
inverse images of the vertices of $P^n$. The $1$-dimensional tori
$\lambda(F_i)$ are uniquely represented by
$$\lambda(F_i)=\{(e^{2\pi i\lambda_{i1}},\ldots, e^{2\pi i\lambda_{in}})\}$$
for $i=1,\ldots, m$, where
$\lambda_i=(\lambda_{i1},\ldots,\lambda_{in})$ is a primitive vector
of $\mathbb Z^n$, that is $\lambda_i$ can be extended to a basis for
$\mathbb Z^n$. Since $P^n$ is a simple polytope, every face is given
by an unique intersection of codimension-$1$ faces, the vectors
$(\lambda_{i1},\ldots,\lambda_{ik})$ are linearly independent for
$k\leq n$ if $F_{\sigma}$ is a face of $P^n$.

Consider the family of codimension-$2$ submanifolds
$M(i)=\pi^{-1}(F_i)$. Define $$M(\sigma) = M(i_0)\cap \ldots\cap
M(i_k)$$ where $\sigma=(i_0,\ldots,i_k)$. Since this intersection is
tranverse,  $M(\sigma) = \pi^{-1}(F_\sigma)$ is a
codimension-$2(k+1)$ submanifold of $M^{2n}$.

The $(m\times n)$-matrix $\wedge=(\lambda_{ij})$ gives an
epimorphism $\lambda:T^m\to T^n$ such that if $$T^m_{\sigma}
=\{(t_1,\ldots,t_m) | \ t_i=1, \  \hbox{if} \ i\notin \sigma \}$$ is
a natural subgroup of $T^m$, then $\lambda|T^m_\sigma:T^m_\sigma\to
\lambda(F_\sigma)$ is an isomorphism. The kernel of $\lambda$
denoted $\hbox{Ker}(\lambda)$, acts freely on $Z(K;(D^2,S^1)$, where
$K$ is the dual of $\partial P^n$ and there is a map
$$Z(K;(D^2,S^1))/\hbox{Ker}(\lambda)\to M^{2n}$$ which is a diffeomorphism.
Furthermore, the epimorphism $\lambda:T^m \to T^n$ is split so that
$T^m$ is isomorphic to $\hbox{Ker}(\lambda) \times T^n$ as a
topological group. A direct consequence is recorded next for use
below.

\begin{lem} \label{lem:fuss.budget.lemma}
The induced map $$\rho: ET^m\times_{T^m} Z(K;(D^2,S^1)) \to
ET^n\times_{T^n} M^{2n}$$ is a homeomorphism.

\end{lem}

\begin{proof}
Observe that $\lambda:T^m\to T^n$ induces a principal fibration
$$\hbox{Ker}(\lambda) \to Z(K;(D^2,S^1))\to M^{2n}.$$
Next, observe that $T^m$ splits as $T^n \times \hbox{Ker}(\lambda)$
as a topological group. Notice that (i) $T^m$ acts on
$Z(K;(D^2,S^1))$, and (ii) the map
$$Z(K;(D^2,S^1))/\hbox{Ker}(\lambda)\to M^{2n}$$ is a diffeomorphism.
Thus $T^m$ acts on $Z(K;(D^2,S^1))$ with $$ET^m\times_{T^m}
Z(K;(D^2,S^1))$$ homeomorphic to $$ET^n\times_{T^n}
Z(K;(D^2,S^1))/\hbox{Ker}(\lambda).$$ The Lemma follows.
\end{proof}

This setting is clearer with the following equivalent construction
of $M^{2n}$ and $Z(K;(D^2,S^1))$ due to Davis and Januzkiewicz in
\cite{davis.jan}. Here, $M^{2n}$ is the quotient of $T^n\times
P^n/\sim$ where $\sim$ is the equivalence relation:
$$(t,x)\sim (t',x')\quad{\rm if}\quad x=x'\quad{\rm and}\quad t't^{-1}\in\lambda(F_\sigma),$$
where $F_\sigma$ is the unique face such that $x\in$
Int$(F_\sigma)$. Now $Z(K;(D^2,S^1))\approx T^m\times P^m/\sim$,
where $\sim$ is the equivalence relation $$(t,x)\sim (t',x')$$ if
$x=x'$ and $t't^{-1}\in T^m_\sigma$ where $x$ lies in the interior
of the unique face, $F_{\sigma}$.

Since the map $$\rho: ET^m\times_{T^m} Z(K;(D^2,S^1)) \to
ET^n\times_{T^n} M^{2n}$$ is a homeomorphism by Lemma
\ref{lem:fuss.budget.lemma}, and is equivariant with respect to the
natural $T^m$-action, it follows that there is a morphism of
fibrations (up to homotopy) as given next.
\[\begin{CD}
\hbox{Ker}(\lambda)  &@>>>  &E(\hbox{Ker}(\lambda)) &@>>>&B(\hbox{Ker}(\lambda)) \\
@VV{ }V &&@VVV&&@VVV\\
Z(K; (D^2,S^1))&@>j_1>> &ET^m\times_{T^m} Z(K;(D^2,S^1))&
@>\pi_1>>&BT^m \\
@VV{ }V&&    @VV{\rho }V    &&@VV B\lambda V\\
M^{2n}&@>j_2>> &ET^n\times_{T^n} M^{2n}&@>\pi_2>>&BT^n
\end{CD}\]

One consequence of this diagram of fibrations and earlier remarks is
the following Theorem concerning the cohomology of toric
manifolds,\cite{buchstaber.panov.2} or \cite{Danilov}.

\begin{thm}\label{thm:cohomology.toric.manifolds}
The cohomology ring $H^*(M^{2n};\mathbb Z)$ is isomorphic to
$$\mathbb Z[x_1,\ldots,x_m]/(I(K)+I(\lambda))$$ where $I(\lambda)$ is
the ideal generated by $\pi_2^*(H^2(BT^n))$.
\end{thm}

\begin{proof}

First recall that (i) the space $E(\hbox{Ker}(\lambda))$ is
contractible, (ii) the sequence of maps
\[
\begin{CD}
E(\hbox{Ker}(\lambda))   @>{}>>   ET^m\times_{T^m} Z(K;(D^2,S^1))
@>{\rho}>> ET^n\times_{T^n} M^{2n}
\end{CD}
\] is a fibration, and (iii) the map
$$\rho: ET^m\times_{T^m} Z(K;(D^2,S^1)) \to ET^n\times_{T^n} M^{2n}$$
is a homotopy equivalence. Consider the composite map
\[
\begin{CD}
ET^m\times_{T^m} Z(K;(D^2,S^1)) @>{\rho}>> ET^n\times_{T^n} M^{2n}
@>{\pi_2}>> BT^n
\end{CD}
\] and observe that this composite $$\pi_2 \circ \rho$$ is again a fibration with fibre $M^{2n}$
as $\rho$ is an equivalence.

By appealing to Morse theory as in \cite{davis.jan}, it follows that
$M^{2n}$ has cells only in even dimensions. The $E_2$-term of Serre
spectral sequence for the fibration
\[
\begin{CD}
M^{2n}  @>{\pi_2 \circ \rho}>> ET^m\times_{T^m} Z(K;(D^2,S^1))
@>{\pi_2 \circ \rho}>>  BT^n
\end{CD}
\] is given by $$H^*( BT^n) \otimes H^*(M^{2n}).$$ Since all
cohomology groups are concentrated in even degrees, the spectral
sequence collapses at $E_2$. It follows that the cohomology of
$H^*(M^{2n})$ is the quotient of the cohomology ring of
$ET^m\times_{T^m} Z(K;(D^2,S^1))$ modulo the two-sided ideal
generated by the image in reduced cohomology of the map
\[
\begin{CD}
\bar{H}^*(BT^n) @>{(\pi_2 \circ \rho)^*}>>\bar{H}^*(ET^m\times_{T^m}
Z(K;(D^2,S^1))).
\end{CD}
\]

Since $\rho$ is a homotopy equivalence ( in fact a homeomorphism ),
the previous statement is equivalent to the statement that the
cohomology of $H^*(M^{2n})$ is the quotient of the cohomology ring
of $ET^n\times_{T^n}Z(K;(D^2,S^1))$ modulo the two-sided ideal
generated by the image in reduced cohomology of the map
\[
\begin{CD}
\bar{H}^*(BT^n) @>{\rho^*}>>\bar{H}^*(ET^m\times_{T^m}
Z(K;(D^2,S^1))).
\end{CD}
\]

It is proven in \cite{buchstaber.panov.2} and \cite{denham.suciu}
that the natural map
$$Z(K;(\mathbb C \mathbb P^\infty, *)) \to
ET^m\times_{T^m}Z(K;(D^2,S^1))$$ is a homotopy equivalence, a point
addressed in Theorem \ref{thm:denham.suciu.BG} here. By
\cite{denham.suciu}, or Theorem \ref{thm:denham.suciu.BG} here, the
cohomology ring of the space $Z(K;(\mathbb C \mathbb P^\infty, *))$
is precisely the Stanley-Reisner ring of $K$,
$$\mathbb Z[x_1, \ldots,x_m]/I(K)$$ where $I(K)$ denotes the
ideal generated by $x_{I}= x_{i_1}\ldots x_{i_k}$ for which $I=
(i_1,\ldots, i_k)$ is not a simplex in $K$.

Thus the cohomology ring of $M^{2n}$ is precisely the ring $\mathbb
Z[x_1, \ldots,x_m]/I(K)$ modulo the two-sided ideal $I(\lambda)$
generated by the image of the map $$\rho^*: \bar{H^*}(BT^n) \to
\mathbb Z[x_1, \ldots,x_m]/I(K).$$

The Theorem follows.

\end{proof}

\section{Preparation} \label{sec:Preparation}

The proof given below for the stable decompositions of moment-angle
complexes relies on work of Welker-Ziegler-$\check{{\rm
Z}}$ivaljevi\'c  on homotopy colimits of spaces \cite{
welker.ziegler.zivaljevic}. Relevant features of that work are
recalled next for purposes of exposition.
\begin{enumerate}
\item Let $CW_*$ denote the category of $CW$-complexes and pointed, continuous
maps.

\item Let $P$ denote a finite poset (partially ordered set).
\end{enumerate}

Given a poset $P$, define a diagram $D$ of CW-complexes over $P$ to
be a functor $$D: \quad P \longrightarrow CW_*$$ so that for every
$p < p'$ in $P$ there is a map
$$\ds{d_{pp'}: \quad D(p') \longrightarrow D(p)}$$ for which
$\ds{d_{pp}}$ is the identity and (redundantly) for $p \geq p' \geq
p''$. Thus, it follows that  $$d_{pp'}\circ d_{p'p''} = d_{pp''}.$$

The colimit of $D$ is the space
$$\mbox{colim}(D) = {\coprod_{p \in P}D(p)}/\! \sim$$ for which
$\sim$ denotes the equivalence relation generated by requiring that
for each $x \in D(p')$, $x \sim d_{pp'}(x)$ for every $p < p'$.

Recall the poset associated to a simplicial complex $K$ with order
given by reverse inclusion of simplices and denoted by the symbol
$\bar{K}$ in Definition
\ref{defin:full.subcomplexes.and.realization}. Given a simplicial
complex $K$, there are two diagrams $D$ and $\hat{D}$ over the poset
$\bar{K}$ given by $D(\sigma)$ and $\hat{D}(\sigma) $ with all the
diagram maps induced by obvious inclusions. It follows that
$$Z(K;(\underline{X},\underline{A})) = \mbox{colim}(D(\sigma))$$ and
$$\widehat{Z}(K; (\underline{X},\underline{A})) = \mbox{colim}(\hat{D}(\sigma)).$$

Homotopy colimits are defined next using the order complex
$\Delta(P)$ as described in the introduction. Given a poset $P$ and
$p \in P$, define a new poset given by
$$P_{\leq p} = \{q \in P: q \leq p\}$$
The order complex $\Delta(P_{\leq p})$ is a sub-complex of
$\Delta(P)$. Also, whenever $p \leq p'$, the natural inclusions
induce an injection $$i_{p'p}: \Delta(P_{\leq p)} \quad
\hookrightarrow \quad \Delta(P_{\leq p'}).$$ For all $p_{1} <
p_{2}$, there are two natural maps $\alpha$ and $\beta$ determined
by $i_{p_{2}p_{1}}$ and $d_{p_{1}p_{2}}$ given by $$1 \times
d_{p_{1}p_{2}} = \alpha: \Delta(P_{\leq p_{1}}) \times D(p_{2})
\longrightarrow \Delta(P_{\leq p_{1}}) \times D(p_{1}),$$ and
$$i_{p_{2}p_{1}} \times 1 = \beta: \Delta(P_{\leq p_{1}}) \times
D(p_{2}) \longrightarrow \Delta(P_{\leq p_{2}}) \times D(p_{2})$$

The homotopy colimit of a diagram $D(P)$ is defined as
$$\mbox{hocolim}(D(P)) = \{\coprod_{p \in P}\Delta(P_{\leq p}) \times D(p)\}/\!\sim$$
where $\sim$ is the equivalence relation defined by identifying
every pair of points $\alpha(x, u)$ and $\beta(x, u)$.

The next theorem, a restatement of the `Projection Lemma' stated in
\cite{welker.ziegler.zivaljevic} as Lemma $3.1$ or in \cite{Segal}
permits replacement of colimits by homotopy colimits in the case of
CW-diagrams (in particular). One feature of this process is that
homotopy colimits frequently have more useful homotopy properties
than colimits in the proofs below.

\begin{thm} \label{thm:hocolim.to.colim}
Let $D(P)$ be a diagram over $P$ having the property that the map
$${\rm colim}_{q > p}\; D(q) \quad \hookrightarrow \quad D(p)$$
is a closed cofibration. Then the natural projection map
$$\pi(D):{\rm hocolim} ((D(P)) \longrightarrow {\rm colim} (D(P)) $$
induced by the projection
$$\Delta(P_{\leq p}) \times D(p) \quad \longrightarrow \quad D(p)$$
is a homotopy equivalence.
\end{thm}

The last ingredient required here is the next result, proven in
\cite{{ziegler.zivaljevic}} as Lemma $1.8$, and stated in
\cite{welker.ziegler.zivaljevic} as Lemma $4.9$ (Wedge Lemma).

\begin{thm} \label{thm:maximal.elements}
Let $P$ be a poset with maximal element $\widehat{1}$. Let $D$ be a
$P$-diagram (a functor $ D:P\rightarrow CW_*$), so that there exists
points $c_{p'} \in D(p')$ for all $p' <\widehat{1}$ such that
$d(p,p')$ is the constant map with $d(p,p')(x)=c_{p'}$ for all $p >
p'$, then there is a homotopy equivalence
$$\underset\longrightarrow {\rm hocolim} (D,P) \longrightarrow \bigvee_{p
\in P} (\Delta (P_{< p})*D(p)).$$
\end{thm}

An additional result is developed next. Let $\Sigma:CW_*\to CW_*$ be
the suspension functor. By the definition of $\hbox{colim}$, if
$D:P\to CW_*$ is a diagram, there is a natural inclusion
$$h(p):D(p)\to {\rm colim}(D(P))$$ which induces a map
$$\Sigma(h(p)):\Sigma(D(p)) \to \Sigma(\rm colim (D(P)))$$ and thus a
map: $$h:{\rm colim}((\Sigma\circ D)(P))\to \Sigma({\rm
colim}(D(P))).$$
\begin{thm}\label{thm:3.3} Let $D:P\to CW_*$ be a diagram of finite CW-complexes, then the map
$$h:{\rm colim}((\Sigma\circ D)(P))\to \Sigma({\rm colim}(D(P)))$$
is a homotopy equivalence.
\end{thm}

\begin{proof} Since homology commutes with finite colimits by a check of the Mayer-Vietoris
spectral sequence, there is a commutative diagram where the vertical
maps are isomorphisms:

\[\begin{CD}
H_*({\rm colim}((\Sigma \circ D)(P)))&@>h_*>>& H_*(\Sigma(\hbox{colim}(D(P)))\\
@V c VV&&   @VV 1 V\\
\hbox{colim}(H_*(\Sigma \circ D(p)))&@>>> &H_*(\Sigma(\hbox{colim}(D(P))) \\
@V s_* VV&&@VV s_* V\\
 \hbox{colim}(H_{*-1}(D(p))) & @>\cong>> & H_{*-1}(\hbox{colim}(D)).
\end{CD} \]

Here $s_*$ is the inverse of the suspension isomorphism in homology,
and $c$ is the commutation isomorphism. Thus $h_*$ is an isomorphism
in homology. Since both $\hbox{colim}((\Sigma\circ D)(P))$ and
$\Sigma(\hbox{colim}((D(P)))$ are both $1$-connected,
J.~H.~C.~Whitehead's theorem gives that $h$ is a homotopy
equivalence.

\end{proof}

The following `Homotopy Lemma' given in \cite{porter}, \cite{vogt}
and \cite{welker.ziegler.zivaljevic} is useful in what follows
below.
\begin{thm} \label{T:3.4} Let $D$ and $E$ be two diagrams over $P$ and with
values in $CW_*$. Then if $f$ is a map of diagrams over $P$ such
that for every $p, f_p:D(p)\to E(p)$ is a homotopy equivalence, then
$f$ induces a homotopy equivalence:
$$\overline{f}: {\rm hocolim} (D(P)) \to {\rm hocolim} (E(P))$$
\end{thm}

\section{Proof of Theorem \ref{thm:decompositions.for.general.moment.angle.complexes} }
Recall the functor $D:K \to \hbox{CW}_*$ of Definition
\ref{defin:moment.angle.complex} as follows: For every $\sigma$ in
$K$, let
$$
D(\sigma) =\prod^m_{i=1}Y_i,\quad {\rm where}\quad
Y_i=\left\{\begin{array}{lcl}
X_i &{\rm if} & i\in \sigma\\
A_i &{\rm if} & i\in [m]-\sigma.
\end{array}\right.$$ Then $Z(K;(\underline{X},\underline{A}))=\bigcup_{\sigma \in K}
D(\sigma)= \mbox{colim } D(\sigma).$

The analogous functor with values in smash products of spaces is
given in Definition \ref{defin:smash.product.moment.angle.complex}
by $\widehat{D}:K \to \hbox{CW}_*$ where
$$
\widehat{D}(\sigma) =Y_1 \wedge Y_2 \wedge \ldots \wedge Y_m \quad
{\rm where}\quad Y_i=\left\{\begin{array}{lcl}
X_i &{\rm if} & i\in \sigma\\
A_i &{\rm if} & i\in [m]-\sigma.
\end{array}\right.
$$

Recall from Definition \ref{defin:smash.products} that $I$ denotes
an ordered sequence $I = (i_1, \ldots, i_k)$ which satisfies $1 \leq
i_1 <\ldots < i_k \leq m$. Specialize $K$ to $K_I$ the full
subcomplex of $K$ given by $$K_I = \{ \sigma \cap I| \ \sigma \in
K\}.$$ Thus $\widehat{D}:K_I \to \hbox{CW}_*$ satisfies
$$
\widehat{D}(\sigma \cap I) =Y_{i_1} \wedge Y_{i_2} \wedge \ldots
\wedge Y_{i_k} \quad {\rm where}\quad
Y_{i_j}=\left\{\begin{array}{lcl}
X_{i_j} &{\rm if} & {i_j}\in \sigma\cap I\\
A_{i_j} &{\rm if} & {i_j}\in I-\sigma\cap I.
\end{array}\right.
$$ Theorem \ref{thm:T2.1} states that there
is a natural pointed homotopy equivalence $$H : \Sigma D(\sigma) \to
\Sigma(\bigvee_{I \subset [m]} \widehat{D}(\sigma \cap I)).$$

\begin{defin}\label{defin:D.Hat}
The diagram $$E:K\to CW_*$$ is defined by the values on objects
$$E(\sigma) = \bigvee_{I \subset [m]} \widehat{D_I}(\sigma \cap I)$$
together with naturality.
\end{defin}

Observe that the map $H$ of Theorem \ref{thm:T2.1} gives a map of
diagrams $H: \Sigma D(\cdot) \to \Sigma E(\cdot)$ which induces a
homotopy equivalence $$H: \Sigma D(\sigma) \to \Sigma E(\sigma)$$
for every $\sigma$. It follows from the `Homotopy Lemma' in
\cite{welker.ziegler.zivaljevic} stated here as Theorem \ref{T:3.4}
that the induced map $$H: {\rm hocolim}(\Sigma D(\cdot)) \to {\rm
hocolim}(\Sigma E(\cdot))$$ is a homotopy equivalence.

The `Projection Lemma' in \cite{welker.ziegler.zivaljevic} stated
here as Theorem \ref{thm:hocolim.to.colim} gives a homotopy
equivalence $$H: {\rm colim}(\Sigma D(\cdot)) \to {\rm colim}(\Sigma
E(\cdot)).$$

Notice that Theorem
\ref{thm:decompositions.for.general.moment.angle.complexes} follows
directly from the features (i) suspension $\Sigma(\cdot)$ commutes
with finite colimits by Theorem \ref{thm:3.3}, (ii) ${\rm
colim}(D(\sigma)) =Z(K;(\underline{X},\underline{A}))$ and (iii)
${\rm colim}(E(\sigma)) = \bigvee_{I \subset[m]}
\widehat{Z}(K;(\underline{X_I},\underline{A_I}))$. Thus there is a
homotopy equivalence $$H:
\Sigma(Z(K;(\underline{X},\underline{A})))\to \Sigma(\bigvee_{I \leq
[m]} \widehat{Z}(K_I;(\underline{X_I},\underline{A_I})))$$ which is
natural for morphisms of pairs $(\underline{X},\underline{A})$ and
embeddings in $K$.

\section{Proof of Theorem \ref{T:1.2} }

The hypotheses of Theorem \ref{T:1.2} are that the inclusion maps
$A_i\subset X_i$ are null-homotopic for all $i$.

\begin{defin}\label{defin:D.sigma}
Consider the diagram $$\widehat{D}:\overline{K}\to CW_*$$ where the
underlying set of $\overline{K}$ is $K$, but ordered with respect to
reverse inclusion where
$$
\widehat{D}(\sigma)=\bigwedge_{i=m}Y_i\quad{\rm where}\quad
Y_i=\left\{\begin{array}{lll}
X_i&{\rm if}&i\in \sigma\\
A_i&{\rm if}& i\in [m]-\sigma\end{array}\right.
$$ with $\widehat{d}_{\sigma,\tau}:\widehat{D}(\tau)\to\widehat{D}(\sigma)$
are given by the natural inclusions when $\sigma<\tau$ (i.e.
$\tau\subset \sigma$).

\end{defin}

The inclusion maps $i_k:A_k\to X_k$ are all null-homotopic. Consider
the diagram $$A_k\stackrel{i_k}{\to}X_k\stackrel{id}{\to}X_k.$$ By
hypothesis, $i_k$ is null-homotopic, so there is a null-homotopy
$F_k:A_k\times I\to X_k$ with $F_k(a,0)=i_k(a)$ and
$F_k(a,1)=f_k(a)=x_k$ for all $a\in A_k$. By the homotopy extension
theorem, there exists $G_k:X_k\times I\to X_k$ with $G_k(x,0)=x,\,\,
G_k(x,1)=g_k(x)$ such that $g_k(a_k) =i_k(a_k)=x_k$ for all $a_k \in
A_k$. Thus there is a commutative diagram

\[\begin{CD}
A_k &@>id>>& A_k\\
@V{i_k}VV &&@V\overline{i}_kVV\\
X_k &@>g_k>> &X_k
\end{CD}\] where $\overline{i}_k$ is the constant map to $x_k$ the base point
of $X_k$.

\begin{defin}\label{defin:E.sigma.i.k.null}
For every $\sigma\in K$, define
$$\widehat{E}(\sigma)
=\left\{\begin{array}{lcl}
\underline{*} &{\rm if} & \sigma \neq \emptyset\\
\widehat{A}^{[m]} &{\rm if} & \sigma = \emptyset.
\end{array}\right.
$$
\end{defin}

Next, define maps $\alpha(\sigma)$ as follows.
\begin{defin}\label{defin:E.sigma}
For every $\sigma\in K$, define
$$\alpha(\sigma):\widehat{D}(\sigma)\to \widehat{E}(\sigma)$$ by
$\alpha(\sigma)=\alpha_1(\sigma)\wedge\ldots\wedge\alpha_m(\sigma)$
where
$$\alpha_k(\sigma)=\left\{\begin{array}{lll}
g_k:X_k\to X_k&{\rm if}&k\in\sigma\\
id:A_k\to A_k&{\rm if}&k\in[m]-\sigma
\end{array}\right.
$$
\end{defin}

Since the $g_k$ are homotopy equivalences, so is $\alpha(\sigma)$
for all $\sigma\in K$. Furthermore, if $ \sigma<\tau$ (i.e.
$\tau\subset\sigma$) the following diagram commutes:

\[\begin{CD}
\widehat{D}(\tau)&@>\alpha(\tau)>>&\widehat{E}(\tau)\\
@V\widehat{d}_{\sigma ,\tau}VV && @VV\widehat{e}_{\sigma,\tau}V\\
\widehat{D}(\sigma)& @>\alpha(\sigma)>> &\widehat{E}(\sigma)
\end{CD}\]

Hence the maps $\alpha(\sigma)$ give a map of diagrams
$$\alpha:\widehat{D}(\overline{K})\to \widehat{E}(\overline{K})$$
which by Theorem \ref{T:3.4} induces a homotopy equivalence
$${\rm hocolim }(\alpha): {\rm hocolim }(\widehat{D}(\overline{K}))\to {\rm hocolim}(\widehat{E}(K)).$$

By naturality, there is a commutative diagram
\[\begin{CD}
{\rm hocolim }(\widehat{D}(\overline{K}))  @>{{\rm hocolim }(\alpha})>> {\rm hocolim}(\widehat{E}(K))\\
@V{\pi(\widehat{D})}VV  @VV{\pi(\widehat{E})}V\\
{\rm colim}(\widehat{D}(K))  @>{{\rm colim}(\alpha)} >> {\rm
colim}(\widehat{E}(K))
\end{CD}\] where $\pi(\widehat{D})$ and $\pi(\widehat{E})$ are the natural projection map.
By Theorem \ref{thm:hocolim.to.colim}, the maps
$$\pi(\widehat{D}):{\rm hocolim }(\widehat{D}(\overline{K})) \to {\rm
colim }(\widehat{D}(\overline{K})), $$ and $$\pi(\widehat{E}):{\rm
hocolim }(\widehat{E}(\overline{K})) \to {\rm colim
}(\widehat{E}(\overline{K}))$$ are homotopy equivalences. Thus, the
induced map $${\rm colim}(\alpha):{\rm
colim}(\widehat{D}(\overline{K}))\to {\rm
colim}(\widehat{E}(\overline{K}))$$ is a homotopy equivalence.

By Theorem \ref{thm:maximal.elements}, there is a homotopy
equivalance $$\hbox{colim}(\widehat{E}(\overline{K}))\to
\bigvee_{p\in P}|\Delta(P_{<p})|*D(p).$$ By definition, there is a
homeomorphism $$\hbox{colim}(\widehat{D}(\overline{K}))\to
\widehat{Z}(K;(\underline{X},\underline{A})).$$

Theorem \ref{T:1.2} follows.

\section{Proof of Theorem \ref{T2.19}}\label{sec:thm:contractible.X.with.zero.skeleton}

Let $(X_i,A_i,x_i)$ denote a triple of $CW$-complexes with
base-point $x_i$ in $A_i$ for $1 \leq i \leq m$ such that all of the
$X_i$ and $A_i$ are closed CW-complexes with all $X_i$ contractible.
Consider the diagram defined next.
\begin{defin}\label{defin:X.i.contractible}
Define $$\hat{D}_{}: \overline{K} \longrightarrow CW_*$$ by
$$\widehat{D}(\omega)
 = Y_1 \wedge \ldots \wedge Y_m,\quad {\rm with}\quad
Y_i=\left\{\begin{array}{lcl}
X_i &{\rm if} & i\in \omega\\
A_i &{\rm if} & i\notin \omega.
\end{array}\right.$$
\end{defin}

Observe that there is an induced homotopy equivalence
$$\mbox{hocolim}(\hat{D}_{}) \to \mbox{colim}(\hat{D}_{}) =
\widehat{Z}(K_{}; (\underline X,\underline{A}))$$ by Theorem
\ref{thm:hocolim.to.colim}.

Define a new diagram as follows.

\begin{defin}\label{defin:new.X.i.contractible}
Define a diagram $$\hat{F}_{}: \overline{K} \longrightarrow CW_*$$
by
$$\widehat{F}(\omega)
 =\left\{\begin{array}{lcl}
 \widehat{A}^{[m]} & \text{ if \; $w =
\emptyset$}\\
* & \text{if \;\, $w \neq \emptyset$}
\end{array}\right.$$for which all maps $$d_{\omega\omega'}:
\widehat{F}(\omega') \longrightarrow \widehat{F}(\omega)$$ are
constant.
\end{defin}

Observe that there is a map of diagrams
$$ \Phi_{\widehat{D}\widehat{F}} : \quad \hat{D}_{}  \longrightarrow \hat{F}_{}$$ obtained
by requiring that $\widehat{D}(\emptyset)$ is sent to
$\widehat{F}(\emptyset)$ by the identity, and where every other
space $\widehat{D}(\omega)$ is sent to the base-point by the
constant map. Since each $X_i$ is contractible,
$\Phi_{\widehat{D}\widehat{F}}$ is a homotopy equivalence on each
object in the diagram. By Theorem \ref{T:3.4} or the Homotopy Lemma
of \cite{welker.ziegler.zivaljevic} ( stated there as Lemma $4.6$ )
there is a homotopy equivalence
$$\mbox{hocolim}(\hat{D}_{}) \to \mbox{hocolim}(\hat{F}_{}).$$

Since there is a homotopy equivalence
$$\mbox{hocolim}(\hat{D}) \to \mbox{colim}(\hat{D}) = \widehat{Z}(K; (\underline X,\underline{A})),$$
it follows that there are homotopy equivalences $$ \widehat{Z}(K;
(\underline X,\underline{A})) \to \mbox{hocolim}(\hat{D}) \to
\mbox{hocolim}(\hat{F}).$$

Notice that ${\ds \hat{F}_{}(\omega) = *}$, unless $\omega =
\emptyset$. In the case $\omega = \emptyset$, $$
\Delta((\overline{K})_{< \emptyset}) = |K|,$$ or, more precisely,
the barycentric subdivision of $|K|$ with
$$\hat{F}_{}(\emptyset) = \widehat{A}^{[m]}.$$ Thus, there are homotopy equivalences
$$\widehat{Z}(K; (\underline{X},\underline{A})) \to  |K|* \widehat{A}^{[m]}.$$ Theorem \ref{T2.19} follows.

\section{Proof of Theorem \ref{thm:null.homotopy.for.A.to.X}}

The conclusion of Theorem \ref{thm:null.homotopy.for.A.to.X} is that
there is a homotopy equivalence $$\Sigma Z(K;( {\underline
X},\underline{A})) \to \Sigma( \bigvee_{I \notin K} |K_{I}| *
\widehat{A}_I).$$ Observe that this conclusion follows from Theorems
\ref{thm:decompositions.for.general.moment.angle.complexes} and
\ref{T2.19} together with the fact that if $I \leq [m]$ corresponds
to a simplex in $K$, then $|K_I|$ is contractible. That is, there
are homotopy equivalences
$$\Sigma(Z(K;(\underline{X},\underline{A})))\to \Sigma(\bigvee_{I
\leq [m]} \widehat{Z}(K_I;(\underline{X}_I,\underline{A}_I)) \to
\Sigma(\bigvee_{I \leq [m]}|K_{I}| *\widehat{A}^{I})$$ with $
\Sigma(\bigvee_{I \leq [m]}|K_{I}| *\widehat{A}^{I})
=\Sigma(\bigvee_{I \notin K}|K_{I}| *\widehat{A}^{I})$.

\section{Proof of Theorem \ref{thm:cohomology.for.contractible.A}}\label{sec:thm:cohomology.for.contractible.A}
Consider the natural inclusion $$Z(K; (\underline{X},\underline{A}))
\to X_1 \times \ldots \times X_m = X^{[m]}$$ together with $I =
(i_1, \ldots, i_t)$ for $ 1 \leq i_1 <\ldots < i_t \leq m$. Recall
the notational convention of Definition \ref{defin:smash.products}
that $\widehat{X}^{I} = X_{i_1}\wedge X_{i_2}\wedge \ldots \wedge
X_{i_t}$. Since all $A_i$ are contractible by assumption, apply
Theorems \ref{thm:decompositions.for.general.moment.angle.complexes}
and \ref{thm:contractible.A} to obtain a homotopy equivalence
$$\Sigma( Z(K; (\underline{X},\underline{A}))) \to  \Sigma( \bigvee_{I \in K} \widehat{X}^{I}).$$
By Theorem \ref{thm:T2.1}, there is a homotopy equivalence
$$\Sigma(X_1 \times \ldots \times X_m) \to \Sigma( \bigvee_{I \in [m]} \widehat{X}^{I}).$$

By naturality, the map $$\Sigma(Z(K; (\underline{X},\underline{A})))
\to \Sigma(X_1 \times \ldots \times X_m)$$ is split with cofibre
$$\Sigma( \bigvee_{I \notin K} \widehat{X}^{I}).$$ Hence there is a split cofibration
$$\Sigma( \bigvee_{I \in K}\widehat{X}^{I}) \to \Sigma( \bigvee_{I \in
[m]} \widehat{X}^{I}) \to \Sigma( \bigvee_{I \notin K}
\widehat{X}^{I}).$$

Next assume that coefficients are taken in a ring $R$ for which
either
\begin{enumerate}
\item $R$ is a field,
\item the cohomology of the $\widehat{X}^{I}$ with coefficients in $R$ satisfies the strong form
of the K\"unneth Theorem (with one case given by $R = \mathbb Z$ for
spaces $X_i$ of finite type with torsion free integral cohomology),
or
\item cohomology is taken in any reduced theory $E^*(\widehat{X}^{I})$ which satisfies the strong form
of the K\"unneth Theorem.
\end{enumerate}

Thus the reduced cohomology $\bar H^*(X^{[m]};R)$ is isomorphic to
the direct sum $$ \oplus_{I \in [m]} \bar{H}^*(X_{i_1};R) \otimes
\cdots \otimes \bar{H}^*(X_{i_t};R).$$ The natural inclusion map
$\Sigma(Z(K; (\underline{X},\underline{A}))) \to \Sigma(X_1 \times
\ldots \times X_m)$ induces a map $$\bar{H}^*(X^{[m]};R) \to \bar
{H}^*(Z(K; (\underline{X},\underline{A}));R)$$ which corresponds to
the projection map
$$ \oplus_{I \in [m]} \bar{H}^*(X_{i_1};R) \otimes \cdots \otimes
\bar{H}^*(X_{i_t};R) \to \oplus_{I \in K} \bar{H}^*(X_{i_1};R)
\otimes  \cdots \otimes \bar{H}^*(X_{i_t};R)
$$ with kernel exactly
$$\oplus_{I \notin K} \bar{H}^*(X_{i_1};R) \otimes  \cdots
\otimes \bar{H}^*(X_{i_t};R).$$ The kernel is exactly the
Stanley-Reisner ideal $I(K)$ by inspection and Theorem
\ref{thm:cohomology.for.contractible.A} follows.

\section{Proof of Theorem \ref{thm:simplicial.moment.angle.15.march.2007}}
\label{sec:simplicial.moment.angle}

The first preparatory proof in this section is that of Lemma
\ref{lem:simplicial.moment.angle} which follows from the fact that
$Z(K;(\underline{X},\underline{A}))$ is natural for morphisms with
respect to morphisms  of pairs,  $(\underline{X},\underline{A}) \to
(\underline{X}^{\prime},\underline{A}^{\prime})$.

\begin{proof}
The first part of the lemma is to check that
\begin{enumerate}
  \item $d_i: Z(K; (\underline{X}_*,\underline{A}_*) \to Z(K; (\underline{X},\underline{A})_*)_{n-1}$ and
  \item $s_j: Z(K; (\underline{X}_*,\underline{A}_*) \to Z(K; (\underline{X},\underline{A})_*)_{n+1}.$
\end{enumerate}

So it suffices to check that if $(x_1,x_2, \ldots, x_m)$ is in the
subset $$Z(K;(\underline{X}_*,\underline{A}_*))_n \subset (X_n(1)
\times \ldots \times X_n(m),$$ then

\begin{enumerate}
\item $d_i(x_1,x_2, \ldots, ,x_m) = (d_i(x_1),d_i(x_2), \ldots, d_i(x_m)) $ is in
$Z(K;(\underline{X}_*,\underline{A}_*))_{n-1}$, and
\item $s_j(x_1,x_2, \ldots, x_m) = (s_j(x_1),s_j(x_2), \ldots, s_j(x_m))$ is in $Z(K;(\underline{X}_*,\underline{A}_*))_{n+1}$
\end{enumerate} Notice that a point $(x_1,x_2, \ldots, ,x_m)$ is in $Z(K;
(\underline{X}_*,\underline{A}_*)) \subseteq X_n(1) \times \ldots
\times X_n(m)$ provided $x_t \in A_n(t)$ for $\sigma \in K$ and
$t\notin \sigma$. But then $d_i(x_t) \in A_{n-1}(t)$ for $ t \notin
\sigma $ as $A_*(t)$ is a simplicial sub-complex of $X_*(t)$ by
assumption. Similarly, notice that $s_j(x_t)$ is in $A_{n+1}(t)$ if
and only if $x_t = d_js_j(x_t)$ is in $A_{n}(t)$. The lemma follows.
\end{proof}

\

The next proof is that of Theorem
\ref{thm:simplicial.moment.angle.15.march.2007}.
\begin{proof}
Let $|u_i,x_i|$ denote the equivalence class of a point in the
geometric realization $|X_*(i)|$. The homeomorphism $$\eta: |X_*(1)|
\times \ldots \times |X_*(m)| \to |X_*(1) \times \ldots \times
X_*(m)|$$ is induced by $\eta(|u_1,x_1|, \ldots,|u_m,x_m|) =
|w,s_{I_1}(x_1), \ldots,s_{I_m}(x_m)|$ for choices of $w$ and
$s_{J_i}$ given in \cite{milnor3, may} as well as in the proof of
Theorem 11.5 in \cite{may2}. Notice that $s_{I_j}(x_j)$ is in
$A_*(j)$ if and only if $x_j$ is in $A_*(j)$ by Lemma
\ref{lem:simplicial.moment.angle}. Thus there is a commutative
diagram
\[
\begin{CD}
Z(K;(|\underline{X}|_*,|\underline{A}|_*)) @>{h}>>   |Z(K; \underline{X}_*,\underline{A}_*))|\\
@VV{i}V          @VV{|i_*|}V          \\
|X_*(1)|\times \ldots \times |X_*(m)| @>{\eta}>> |X_*(1) \times
\ldots \times X_*(m)|
\end{CD}
\] as stated in Theorem \ref{thm:simplicial.moment.angle.15.march.2007}.

Since (1) $\eta$ is a homeomorphism, (2) $|i_*|$ is a monomorphism,
and (3) $s_{I_j}(x_j)$ is in $A_*(j)$ if and only if $x_j$ is in
$A_*(j)$, it follows that $h$ is a surjection. Since $h$ is a
continuous, bijective, open map, it is a homeomorphism and the
theorem follows.
\end{proof}

\section{Connection to Kurosh's theorem}

The purpose of this section is to point out that certain fibrations
above involving the polyhedral product functor give natural, useful
covering spaces. For example, Ganea's theorem, a special case of the
Denham-Suciu fibration, implies a statement equivalent to a
classical theorem within group theory due to A.~Kurosh
\cite{kurosh}. This section is purely expository development of this
connection.

Kurosh's theorem addresses the structure of a free product of
discrete groups $$\amalg_{1 \leq \alpha \leq n} G_{\alpha}.$$ One
statement of the theorem is as follows. If $G$ is a subgroup of the
free product of discrete groups $A$ and $B$, $A\amalg B$, then $G$
is a free product of the following groups.
\begin{enumerate}
  \item a subgroup of $G$ conjugate in $A\amalg B$ to a subgroup of $A$,
  \item a subgroup of $G$ conjugate in $A\amalg B$ to a subgroup of $B$, and/or
  \item a free group.
\end{enumerate}

The first observation is the following corollary of Kurosh's
theorem.

\begin{prop} \label{P:7.1} Let $G$ be a subgroup of the free product $A\amalg B$.
Then $G$ surjects to a subgroup of $A\times B$ with kernel a free
group.
\end{prop}

The fibration of Denham-Suciu stated here as Lemma
\ref{denham.suciu.fibrations} provides a simple reformulation of
this corollary by inspection as given next. The homotopy fibre of
the inclusion $$X\bigvee Y \to X \times Y$$ is a generalized
moment-angle complex which has the homotopy type of the join
$$\Omega(X) * \Omega(Y).$$ Furthermore,  the induced map $$W\colon
\Omega(X) * \Omega(Y) \to X\bigvee Y $$ is the Whitehead product
$$[E_X,E_Y]$$ where $E_X: \Sigma \Omega(X) \to X \bigvee Y$ denotes
the composite

$$
\Sigma \Omega(X) \stackrel{e_X}{\to} X \stackrel{i_X}{\to} X \bigvee
Y
$$ with $e_X$ the natural evaluation map and with $i_X$ the natural
inclusion.

Recall that for two pointed topological spaces $U$ and $V$, each of
the homotopy type of a CW-complex, their join $U*V$ has the homotopy
type of the suspension $\Sigma(U \wedge V)$. This is applied next.
Consider the case of $$X = K(\pi,1) \quad \mbox{and} \quad Y =
K(\Gamma,1)$$ for discrete groups $\pi$ and $\Gamma$ to obtain
$$\Omega(X) =  K(\pi,0) \quad \mbox{and} \quad \Omega(Y) = K(\Gamma,0).$$ Thus
$\Omega(X) * \Omega(Y)$ has the homotopy type of $$ \Sigma((\pi
\times \Gamma)/ (\pi \bigvee \Gamma))$$ which is a bouquet of
circles one for each pair $\{(\alpha,\beta) \in \pi \times \Gamma|
\alpha \neq 1, \beta \neq 1\}$. \

The next proposition follows easily.
\begin{prop}
\label{P:7.2} The kernel of the natural projection map $$q: \pi
\amalg \Gamma \to \pi \times \Gamma$$ is the fundamental group of
$\Sigma((\pi \times \Gamma)/ (\pi \bigvee \Gamma))$, a free group.
Furthermore, the Whitehead product map $W\colon \Omega(X) *
\Omega(Y) \to X\bigvee Y $ induces the map from $\pi_1(\Sigma((\pi
\times \Gamma)/ (\pi \bigvee \Gamma)) \to \pi\amalg \Gamma$ which is
an isomorphism onto the kernel of the map $q: \pi \amalg \Gamma \to
\pi \times \Gamma$.

\end{prop}

The proof of Proposition \ref{P:7.1} is given next. Let $G$ be a
subgroup of $A\amalg B$. Project to $A$ and to $B$, then let the
images of $G$ be $G_A$ and $G_B$ respectively. Define the map
$$j_A:G \to A\amalg B$$ as the composite
$$
G \to G_A \stackrel{i_A}{\to} A \stackrel{\gamma_A}{\to} A \amalg B
$$
\noindent where $i_A$ and $\gamma_A$ are natural inclusions. Define
$$j_B:G \to A\amalg B$$ similarly. Thus the kernel of the natural map $$j_A \times j_B: G \to G_A\times
G_B$$ is a subgroup of the kernel of
$${q:A\amalg B \to A\times B},$$ a free group.

A standard consequence is recorded next.
\begin{cor}\label{C:7.3} Let
$$f: G_1 \amalg \cdots \amalg G_n  \to H$$ be a group homomorphism
which when restricted to each $G_i$ is a monomorphism. Then the
kernel of $f$ is free.
\end{cor}

\section{Acknowledgements} \label{sec:Acknowlegements}
The authors would like to thank the Institute for Advanced Study for
providing a wonderful environment to develop this mathematics. The
first author was partially supported by the award of a research
leave from Rider University. The third author was partially
supported by NSF grant number 0340575 and DARPA grant number
2006-06918-01. The fourth author was partially supported by the
Department of Mathematics at Princeton University.

\bibliographystyle{amsalpha}

\begin{thebibliography}{99}


\bibitem{abbcg} A.~Adem, A.~Bahri, M.~Bendersky, F.~R.~Cohen, and S.~Gitler, {\em On stable decompositions for simplicial spaces},
submitted.

\bibitem{astey} L.~Astey, in preparation.

\bibitem{baskakov} I.~Baskakov, {\em Cohomology of K-powers of spaces and the
combinatorics of simplicial divisions}, Russian Math. Surveys 57
(2002), no. 5, 989–990.

\bibitem{baskakov2} \bysame, {\em Triple Massey products in the cohomology of
moment-angle complexes}, Russian Math. Surveys 58 (2003), no. 5,
1039–1041.

\bibitem{baskakov3}   I.~Baskakov, V.~Buchstaber and T.~Panov, {\em  Algebras of cellular cochains,
and torus actions}, Russian Math. Surveys 59 (2004), no. 3, 562–563.

\bibitem{bjorner.sarkaria} A. Bjorner, K.S. Sarkaria {\em The zeta function of a simplicial complex}, Israel Journal of Math. 103(1998) 29-40.

\bibitem{buchstaber.panov} V.~Buchstaber, and T.~Panov,{\em Actions of tori, combinatorial
topology and homological algebra}, Russian Math. Surveys 55 (2000),
no. 5, 825–921.

\bibitem{buchstaber.panov.2} \bysame,{\em  Torus actions and their applications in topology and
combinatorics }, AMS University Lecture Series, volume 24, (2002).

\bibitem{buchstaber.panov.ray} V.~Buchstaber, T.~Panov, N.~Ray,
{\em Spaces of polytopes and cobordism of quasitoric manifolds},
Moscow Math. J. 7 (2007), no.2, 219-242.

\bibitem{cohen.ghrist.koditschek} F.~R.~Cohen, R.~Ghrist, and
D.~Koditschek, {\em Toward vector field planners for simplicial
specification of robot coordination}, in preparation.

\bibitem{Danilov} V.~I.~Danilov, {\em The geometry of toric varieties} (Russian),
Uspekhi Mat. Nauk, 33(2) 1978, pages 85-134, translated in Russian
Matrh Surveys, 33(2), 1978, 97-154.

\bibitem{davis.jan} M.~Davis, and T.~Januszkiewicz, {\em Convex polytopes,
Coxeter orbifolds and torus actions}, Duke Math. J. 62 (1991), no.
2, 417–451.

\bibitem{deconcini.procesi} C.~De~Concini and C.~Procesi, {\em  Wonderful models of subspace
arrangements}, Selecta Math. (N.S.) 1 (1995), no. 3, 459–494.

\bibitem{dgm} P.~Deligne, M.~Goresky, and R.~MacPherson, {\em   L'alg`ebre de
cohomologie du compl´ement, dans un espace affine, d'une famille
finie de sous-espaces affines}, Michigan Math. J. 48 (2000), 121–
136.

\bibitem{denham} G.~Denham, {\em Homology of subgroups of right-angled Artin groups}, arxiv math.GR/0612748.

\bibitem{denham.suciu} G.~Denham and A.~Suciu, {\em Moment-angle complexes, monomial ideals and Massey products},
Pure Appl. Math. Q. 3 (2007), no. 1, 25-60.

\bibitem{franz} M.~Franz, {\em The integral cohomology of toric manifolds},
Proc. Steklov Inst. Math. 252 (2006), 53--62 [Proceedings of the
Keldysh Conference, Moscow 2004].

\bibitem{franz2} \bysame, {\em Koszul Duality for Tori}, Thesis, Universita\"at Konstanz, 2001.


\bibitem{ganea} T.~Ganea, {\em A generalization of the homology and homotopy suspension}, Comment. Math.
Helv., 39(1965), 295-322.

\bibitem{goresky.macpherson} R.~Goresky and R.~MacPherson, {\em Stratified Morse Theory}, Ergeb.
Math. Grenzgeb., vol. 14, Springer-Verlag, Berlin, 1988.

\bibitem{grbic.theriault} J.~Grbi´c and S.~Theriault, {\em Homotopy type of the complement of a coordinate subspace arrangement
of codimension two}, Russian Math. Surveys 59 (2004), no. 3,
1207–1209.

\bibitem{hochster} M.~Hochster, {\em Cohen-Macaulay rings, combinatorics, and simplicial
complexes}, in: Ring theory, II (Proc. Second Conf., Univ. Oklahoma,
Norman, Okla., 1975), pp. 171–223, Lecture Notes in Pure and Appl.
Math., vol. 26, Dekker, New York, 1977.

\bibitem{hu}Y.~Hu, {\em On the Homology of Complements of Arrangements of Sub-spaces and
Spheres}, P.A.M.S., Vol. 122, No. 1 (Sep., 1994), pp. 285-290.

\bibitem{james}I.~James, {\em Reduced product spaces}, Ann. of Math., 62, No. 1 (1955), 170-197.

\bibitem{jewell}K.~Jewell, {\em Complements of sphere and sub-space arrangements}, Topology and its Applications 56 (1994) 199-214.

\bibitem{jewell.orlik.shapiro}K.~Jewell, P.~Orlik and B.~Z.~Shapiro, {\em On the complements of affine sub-sapce
arrangements}, Topology Appl. 56, No.3, 215-233 (1994).

\bibitem{kamiyama.tsukuda}Y. Kamiyama, S. Tsukuda, {\em The configuration space of the $n$-arms machine in the Euclidean space}, Topology and its Applications 154(2007), 1447-1464.

\bibitem{kurosh}A.~Kurosh, {\em Lectures on general algebra}, published in Russian in 1960 and an English translation published in
1963, the Chelsea Publishing Company, New York.

\bibitem{delong.schultz}M.~de~Longueville and C.~Schultz,  {\em The cohomology rings of
complements of subspace arrangements}, Math. Ann. 319 (2001), no. 4,
625–646.

\bibitem{Santiago} S.~Lopez~de~Medrano, {\em   Topology of the intersection of quadrics in $\mathbb R^n$}, in {\it Algebraic Topology}
(Arcata Ca), Springer Verlag LNM  \textbf{1370}(1989), Springer
Verlag.


\bibitem{may} J.~P.~May, {\em The Geometry of Iterated Loop Spaces}, SLM, \textbf{268}(1972),
Berlin-New York: Springer.

\bibitem{may2} \bysame, {\em Simplicial Objects in Algebraic Topology}, Van
Nostrand, Math. Studies, no. 11, 1967.


\bibitem{milnor}J.~Milnor, {\em On the construction F[K]}, In: A student's Guide to Algebraic
Topology, J.F.~Adams, editor, Lecture Notes of the London
Mathematical Society, \textbf{4}(1972), 119-136.


\bibitem{milnor3} \bysame, {\em The geometrical realization of a
semi-simplicial complex}, Ann. of Math., \textbf{65}(1957),
357--362.

\bibitem{notbohm.ray} D.~Notbohm, and N.~Ray, {\em  On Davis-Januszkiewicz homotopy types. I. Formality and rationalization}, AGT
\textbf{5}(2005), 31-51.


\bibitem{panov} T.~Panov, {\em Cohomology of face rings and torus actions}, arXiv:math.AT/0506526.

\bibitem{panov.ray.vogt} T.~Panov, N.~Ray, and R.~Vogt {\em  Colimits, Stanley-Reisner
algebras, and loop spaces}, in: Categorical decomposition techniques
in algebraic topology, Prog. Math., vol. 215, Birkh¨auser, Basel,
2004, pp. 261–291.

\bibitem{papadima.suciu} S.~Papadima and A.~Suciu, {\em  Algebraic invariants for right-angled Artin groups}, Math.
Ann., posted on Dec. 14, 2005 (to appear), also available at
arXiv:math.GR/0412520. 47.


\bibitem{porter} G.~Porter, {\em The homotopy groups of wedges of suspensions}, Amer. J. Math.,
\textbf{88}(1966), 655-663.

\bibitem{Segal} G.~Segal, {\em Classifying spaces and spectral
sequences}, Pub. Math. des I. H. E. S., \textbf{34}(1968), 105-112.

\bibitem{stanley}   R.~P.~Stanley, {\em Combinatorics and commutative algebra}, Second
edition, Progr. Math., vol. 41, Birkh¨auser, Boston, MA, 1996. MR
98h:05001 48.

\bibitem{stone} D.~Stone, {\em Private communication}.

\bibitem{strickland} N.~Strickland, {\em  Notes on toric spaces}, preprint 1999, available at
Strickland's webpage.

\bibitem{swartz}E.~Swartz, {\em Topological representations of matroids
}, JAMS, 16 (2003), 427-442.


\bibitem{vinberg} E.~B.~Vinberg, {\em Discrete linear groups that are generated by
reflections}( Russian ), Izvestia: Math. 35(1971), 1072-1112.

\bibitem{vogt} R.~Vogt, {\em Homotopy limits and colimits}, Math. Zeit. 134 (1973), 11-52.

\bibitem{welker.ziegler.zivaljevic} V.~Welker, G.~Ziegler, R.~\v{Z}ivaljevi\'c, {\em Homotopy colimits-comparison lemmas for
combinatorial applications}, J.~Reine Angew. Math., 509(1999),
117-149.


\bibitem{ziegler.zivaljevic}G.~Ziegler, R.~\v{Z}ivaljevi\'c, {\em Homotopy types of sub-space arrangements via diagrams of
spaces}, Math. Ann. \textbf{295}(1993), 527--548.

\end{thebibliography}

\end{document}